\documentclass{amsart}

\usepackage{latexsym}
\usepackage{graphicx}
\usepackage{booktabs}
\usepackage{amssymb,amsmath,amsthm}
\usepackage{enumerate}
\usepackage{comment}
\usepackage{overpic}

    \DeclareMathOperator*{\supp}{supp}
    \DeclareMathOperator*{\dist}{dist}

    \DeclareMathOperator{\Tr}{Tr}

\newtheorem{thm}{Theorem}[section]

\newtheorem{lem}[thm]{Lemma}
\newtheorem{prop}[thm]{Proposition}
\theoremstyle{definition}

\theoremstyle{remark}
\newtheorem{rem}[thm]{Remark}
\numberwithin{equation}{section}
\title[Recurrence relations and vector equilibrium problems]{Recurrence relations and vector equilibrium problems arising from a model of
non-intersecting squared Bessel paths}

\author{A.B.J. Kuijlaars}
\address{Department of Mathematics, Katholieke Universiteit Leuven, Celestijnenlaan 200b - bus 2400,
3001 Leuven, Belgium.}
\email{arno.kuijlaars@wis.kuleuven.be}

\author{P. Rom\'an}
\address{Department of Mathematics, Katholieke Universiteit Leuven, Celestijnenlaan 200b - bus 2400,
3001 Leuven, Belgium.}
\address{CIEM, FaMAF, Universidad Nacional de C\'ordoba, Medina Allende
s/n Ciudad Universitaria, C\'ordoba, Argentina.}
\email{PabloManuel.Roman@wis.kuleuven.be}

\thanks{
The authors are supported by K.U. Leuven research grant OT/08/33.
The first author is also supported by  FWO-Flanders project
G.0427.09, by  the Belgian Interuniversity Attraction Pole P06/02, by the
European Science Foundation Program MISGAM, and by grant
MTM2008-06689-C02-01 of the Spanish
Ministry of Science and Innovation.}

\begin{document}

\begin{abstract}
In this paper we consider the model of $n$ non-intersecting squared
Bessel processes with parameter $\alpha$,
in the confluent case where all particles start, at time $t=0$,
at the same positive value $x=a$, remain positive, and end, at time $T=t$, at the position $x=0$.
The positions of the paths have a limiting mean density as $n\to\infty$ which is characterized
by a vector equilibrium problem. We show how to obtain this equilibrium problem
from different considerations involving the recurrence relations for multiple orthogonal polynomials
associated with the modified Bessel functions.

We also extend the situation by rescaling the parameter $\alpha$,
letting it increase proportionally to $n$ as $n$ increases. In this case we also analyze
the recurrence relation and obtain a vector equilibrium problem for it.
\end{abstract}

\maketitle

\section{Introduction and statement of results}

\subsection{Introduction}

This paper deals with the model, studied in \cite{KMW}, of $n$ non-intersecting squared
Bessel paths in the confluent limit where all paths start, at time $t=0$,
at the same positive value $x=a$, remain positive, and end at a later time  $t=T$, at the position $x=0$.

The squared Bessel process is a diffusion process on $[0,\infty)$, depending on
a parameter $\alpha > -1$, whose transition probability density is given by
\begin{align*}
    P_t^\alpha (x,y)&= \frac{1}{2t} \left(\frac{y}{x} \right)^{\frac{\alpha}{2}}
    e^{-\frac{x+y}{2t}}I_\alpha\left( \frac{\sqrt{xy}}{t} \right), \quad x,y>0, \\
    P_t^\alpha (0,y)&= \frac{y^\alpha}{(2t)^{\alpha+1}\Gamma(\alpha+1)}e^{-\frac{y}{2t}}, \quad y>0.
\end{align*}
Here, $I_\alpha$ denotes the modified Bessel function of the first kind of order $\alpha$,
\[ I_{\alpha}(z) = \sum_{k=0}^\infty \frac{(z/2)^{2k+\alpha}}{k!\Gamma(k+\alpha+1)}. \]

If we let $n\to \infty$, and perform an appropriate time scaling
$t \mapsto t/(2n)$, $T \mapsto 1/(2n)$, the paths fill out a
region in the $tx$-plane as shown in the left figure in Figure~\ref{figureNon}.
The figure shows a numerical simulation of $50$ non-intersecting
squared Bessel paths and the boundary of the region filled by the
paths. The right part of Figure~\ref{figureNon} shows a similar picture
where in addition $\alpha$ increases proportionally with $n$.
Since $\alpha$ is a measure for the repulsion from $0$ all paths
in the right part of Figure~\ref{figureNon} stay at a positive distance from $0$
until the final time $t=1$. In contrast, in the left part of
Figure~\ref{figureNon}, where $\alpha$ remains fixed, the
smallest paths arrive at the hard edge (the wall) at $x=0$ already
at a critical time $t = t^* = a/(1+a) < 1$, see \cite{KMW}.

\begin{figure}
\centering
\begin{overpic}[width=5.9cm,height=4.72cm]{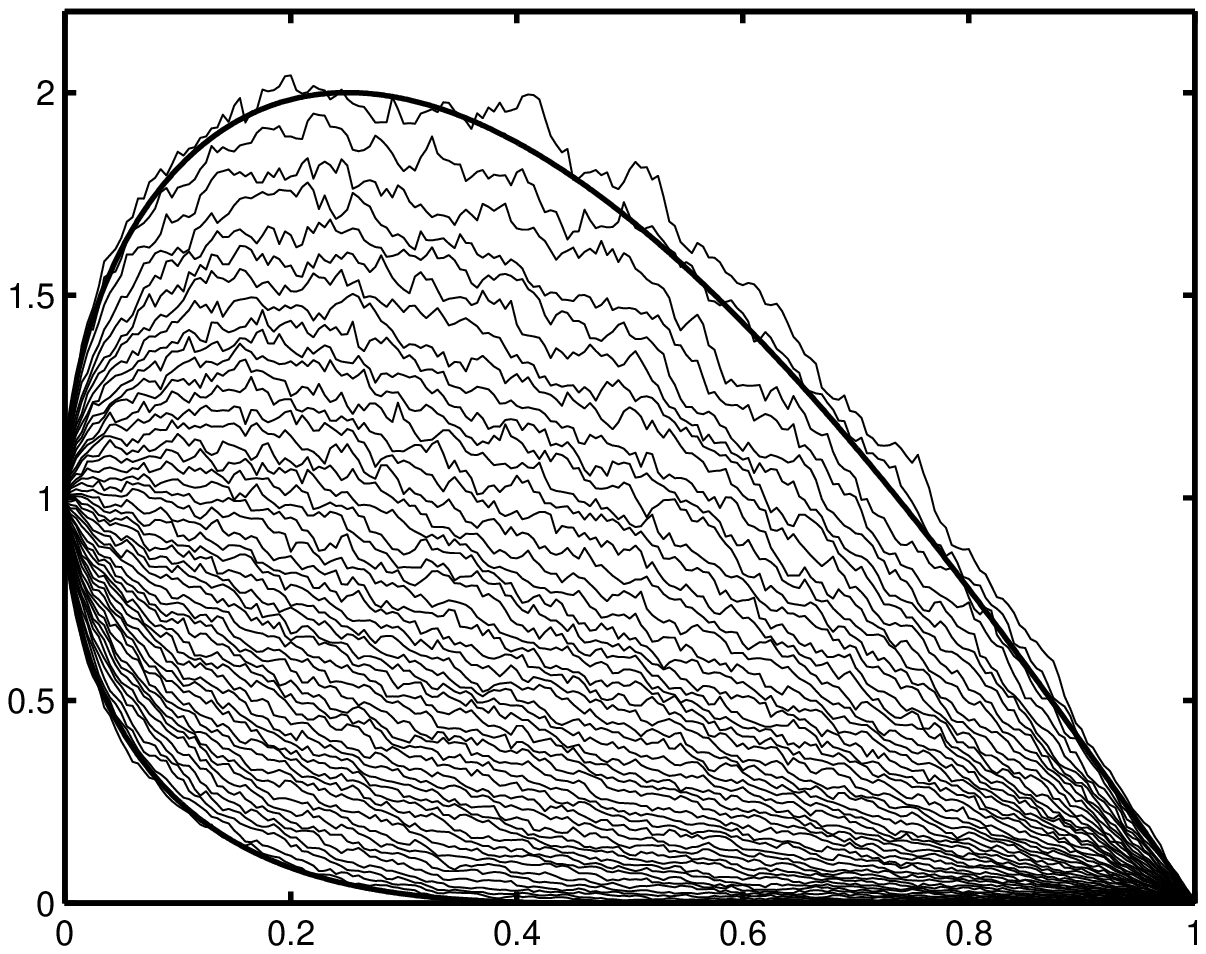}
\put(69,67.5){$a=1$}
\put(69,61.5){$\alpha$ fixed}
\put(69,2){$t$}
\put(5,40){$x$}
\end{overpic}
\hspace{.5cm}
\begin{overpic}[width=5.9cm,height=4.72cm]{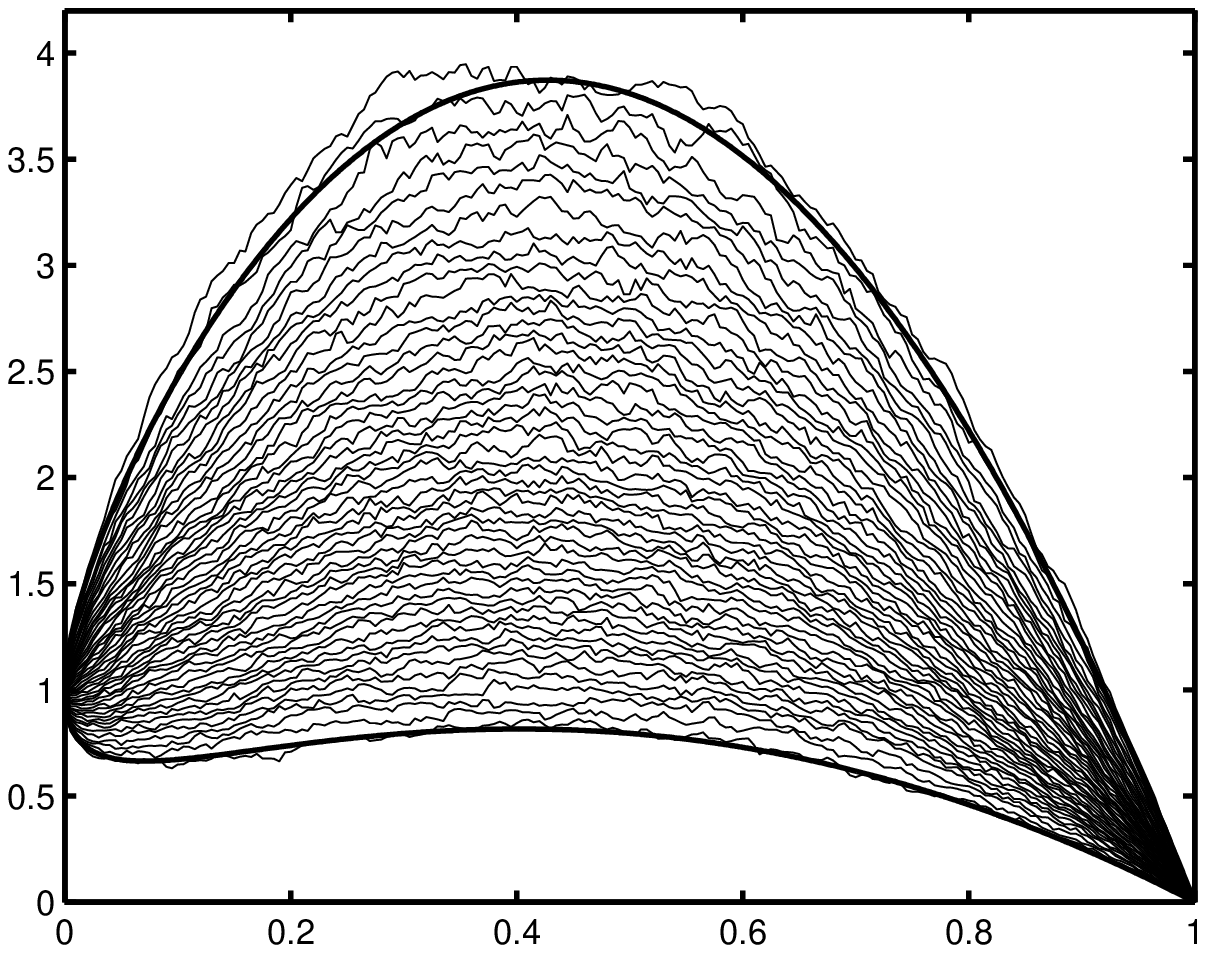}
\put(69,67.5){$a=1$}
\put(69,61.5){$\alpha = 5n$}
\put(50,2){$t$}
\put(5,40){$x$}
\end{overpic}
\caption{ \label{figureNon} Numerical simulation of $n=50$
non-intersecting squared Bessel paths. The left figure shows $n$
paths with fixed $\alpha$, and the right figure shows $n$ paths
with $\alpha$ increasing proportionally with $n$.}
\end{figure}

The model with fixed $\alpha$ was studied in \cite{KMW}. One of the
results was that for every $t \in (0,1)$, the positions of
the paths at time $t$ have a limiting mean density as $n \to \infty$,
that is characterized by the following vector equilibrium problem
(see Theorem 2.4 and the appendix of \cite{KMW}).
The vector equilibrium problem asks to minimize the functional
\begin{multline} \label{eq:energyfunctional}
    \iint \log \frac{1}{|x-y|} d\nu_1(x)d\nu_1(y) + \iint \log \frac{1}{|x-y|}d\nu_2(x)d\nu_2(y)\\
    -  \iint \log \frac{1}{|x-y|} d\nu_1(x)d\nu_2(y) +
        \int \left(\frac{x}{t(1-t)}-\frac{2\sqrt{ax}}{t} \right) d\nu_1(x),
\end{multline}
over all vectors of measures $(\nu_1,\nu_2)$ such that
\begin{align} \label{eq:vectornormalization}
    \supp(\nu_1)\subset [0,\infty), \quad \int d\nu_1 = 1, \qquad
    \supp(\nu_2)\subset (-\infty,0], \quad \int d\nu_2 = 1/2,
    \end{align}
 and
\begin{align} \label{eq:constraint}
    \nu_2 \leq \sigma,
\end{align}
where $\sigma$ is the measure on $(-\infty, 0]$ with density
\begin{align} \label{eq:constraintdensity}
    \frac{d\sigma(x)}{dx} = \frac{\sqrt{a}}{\pi t}|x|^{-1/2}, \qquad x\in (-\infty,0].
    \end{align}
There is a unique minimizer $(\nu_1,\nu_2)$, and the part $\nu_1$ has a density
which is the limiting mean density of the positions of the paths at time $t$.
The equilibrium problem depends on the parameters $a > 0$ and $0 < t < 1$
that appear in the energy functional \eqref{eq:energyfunctional}
as well as in the upper constraint \eqref{eq:constraintdensity}.

The above equilibrium problem is somewhat unusual since it
involves a second measure $\nu_2$ with a constraint $\sigma$. It
turns out (see \cite{KMW}) that the constraint is not active if $t
< t^* = a/(a+1)$, that is before the critical time where the
smallest paths come to the hard edge. The constraint is active for
$t > t^*$.

It is the aim of this paper to give more insight into the nature
of the vector equilibrium problem with constraint. We show how it
arises from different considerations involving the recurrence
relations for multiple orthogonal polynomials associated with
modified Bessel functions, see the next section.

We note that a similar vector equilibrium problem with constraint
and external field was also found in the context of the Hermitian
two-matrix model
 \[ \frac{1}{Z_n} \exp \left(-n \Tr (V(M_1) + W(M_2) - \tau M_1 M_2) \right) dM_1 dM_2 \]
where $W(M) = \frac{1}{4} M^4$, see \cite{DK2}. The vector
equilibrium problem now involves three measures, which are located
on $\mathbb R$, $i \mathbb R$ and $\mathbb R$, respectively, with
a constraint acting on the second measure. The energy functional
is similar to  \eqref{eq:energyfunctional}. The polynomials that
are connected to this model satisfy a five term recurrence
relation.  The ideas that are developed in this paper can be
extended to analyze this recurrence relation and it is possible to
obtain the vector equilibrium problem from it. This will be
reported elsewhere.

\subsection{Multiple orthogonal polynomials}

In \cite{KMW} it was shown that the positions of the paths at any time $t \in (0,T)$ constitute a
multiple orthogonal polynomial (MOP) ensemble associated with the two weight functions
\begin{equation} \label{Besselweights}
    w_{1}(x) = x^{\frac{\alpha}{2}} e^{-\frac{Tx}{2t(1-t)}} I_\alpha
    \left(\frac{\sqrt{ax}}{t}\right),\qquad
    w_{2}(x) =  x^{\frac{\alpha+1}{2}} e^{-\frac{Tx}{2t(1-t)}}I_{\alpha+1} \left(\frac{\sqrt{ax}}{t}\right),
\end{equation}
defined for $x \in (0,+\infty)$. It implies in particular that the
`average characteristic polynomial'
\begin{equation} \label{MOPpolynomial}
    B_n(x) = \mathbb E \left[ \prod_{j=1}^n (x-x_j(t)) \right],
    \end{equation}
where $x_1(t) < x_2(t) < \cdots < x_n(t)$ are the positions of the $n$ paths at time $t$,
satisfies the orthogonality relations (assume $n$ is even)
\begin{equation} \label{MOPorthogonality}
    \int_0^{\infty} B_n(x) x^k w_j(x) dx = 0, \qquad
    \text{for } k=0, 1, \ldots, n/2-1, \quad j=1,2.
    \end{equation}
The polynomial \eqref{MOPpolynomial} is characterized by a $3
\times 3$ matrix valued Riemann-Hilbert problem and the
correlation kernel for the MOP ensemble can be written directly in
terms of the solution of the Riemann-Hilbert problem. See also
\cite{Kuijlaars} for the general notion of a MOP ensemble.

The multiple orthogonal polynomials $B_n$ for the weights \eqref{Besselweights}
were studied in detail  by Coussement and Van Assche \cite{CVA1}, \cite{CVA2}.
They obtained a third order differential equation, which was used in \cite{KMW}.
In addition they also gave an explicit four term recurrence relation
\begin{equation} \label{eq:polynomialsBk}
    x B_k(x) = B_{k+1}(x) + b_k B_k(x) + c_k B_{k-1}(x) + d_k B_{k-2}(x)
    \end{equation}
with recurrence coefficients
\begin{equation} \label{eq:coefficientsbcdk}
\begin{aligned}
    b_k & = \frac{a(T-t)^2}{T^2} + \frac{2t(T-t)}{T}(2k+ \alpha + 1), \\
    c_k & = \frac{4at(T-t)^3}{T^3}k + \frac{4t^2(T-t)^2}{T^2}k(k+\alpha), \\
    d_k & = \frac{4at^2(T-t)^4}{T^4} k(k-1).
    \end{aligned}
    \end{equation}
After the rescaling $t \mapsto t/(2n)$,
$T \mapsto 1/(2n)$, we obtain a doubly indexed sequence $B_{k,n}$ of polynomials
satisfying for each $n$, the recurrence
\begin{equation} \label{eq:polynomialsBkn}
    x B_{k,n}(x) = B_{k+1,n}(x) + b_{k,n} B_{k,n}(x) + c_{k,n} B_{k-1,n}(x) + d_{k,n} B_{k-2,n}(x)
    \end{equation}
with recurrence coefficients
\begin{equation} \label{eq:coefficientsbcdkn}
\begin{aligned}
    b_{k,n} & = a(1-t)^2 + t(1-t) \frac{2k+ \alpha + 1}{n}, \\
    c_{k,n} & = 2at(1-t)^3 \frac{k}{n} + t^2(1-t)^2 \frac{k(k+\alpha)}{n^2}, \\
    d_{k,n} & = at^2(1-t)^4 \frac{k(k-1)}{n^2}.
    \end{aligned}
    \end{equation}
It is the aim of this paper to obtain the vector equilibrium problem from
this recurrence relation. We will show that the measure $\nu_1$ is the limiting zero
distribution of the diagonal polynomial $B_{n,n}$ as $n \to \infty$.

We also consider the following extension of the situation studied in \cite{KMW}.
We rescale the parameter $\alpha$ by letting it increase proportionally to $n$ as $n$ increases.
The right figure in Figure~\ref{figureNon} shows
a numerical simulation of 50 non-intersecting paths for this case. We observe that now the
paths stay away from the hard edge $x=0$. In this case we are also able to analyze the
recurrence relation and obtain a vector equilibrium problem from it.

\subsection{Polynomials satisfying an $m$-term recurrence relation}

The recurrence coefficients \eqref{eq:coefficientsbcdkn} are such
that whenever $n \to \infty$, $k \to \infty$ so that $k/n \to s$
exists, there is a limit
\[ b_{k,n} \to b(s), \qquad c_{k,n} \to c(s), \qquad d_{k,n} \to d(s) \]
for certain functions $b(s)$, $c(s)$ and $d(s)$.

We will study this situation in the general context of polynomials $P_{k,n}$
depending on two parameters, satisfying for each $n$, an $m$-term
recurrence relation with varying coefficients
\begin{equation} \label{recurrence_m}
    xP_{k,n}(x) = P_{k+1,n}(x)+b_{k,n}^{(0)} P_{k,n}(x) + b_{k,n}^{(1)}P_{k-1,n}(x)+
        \cdots+b_{k,n}^{(m-2)}P_{k+2-m,n}(x),
\end{equation}
with $P_0\equiv 1$, $P_{-1}\equiv 0,\ldots,P_{-m+2}\equiv 0$, where the
recurrence coefficients have scaling limits
\[ \lim_{k/n\to s} b_{k,n}^{(j)} = b^{(j)}(s), \qquad j=0, \ldots, m-2. \]
for certain functions $b^{(0)}, \ldots, b^{(m-2)}$. The notation $\lim_{k/n \to s}$
that we use here and also later in the paper, means that both $k, n \to \infty$ with $k/n \to s$.

We associate with the functions $b^{(j)}$, $j=0,\ldots, m-2$, a family of functions
\begin{equation} \label{sympbolsAs}
    A_s(z) = z + b^{(0)}(s) + b^{(1)}(s)z^{-1}+ \cdots +b^{(m-2)}(s) z^{-m+2}
\end{equation}
and the sequence of Toeplitz matrices $(T_n)_n$ where $T_n = T_n(A_s)$
is the $n \times n$ Toeplitz matrix with symbol $A_s$, defined by
\begin{equation}
\label{definition_toeplitz_matrix}
    (T_n(A_s))_{jk}=\begin{cases} 1,          & \text{if }k=j+1,\\
                              b^{(i)}(s), & \text{if }k=j-i, \qquad i=0, \ldots, m-1,\\
                              0,          & \text{otherwise}. \end{cases}
\end{equation}

The limiting behavior of the spectrum of $T_n(A_s)$  as $n\to
\infty$ is characterized in terms of the solutions of the
algebraic equation $A_s(z)=x$, see \cite{BG}. For every
$x\in\mathbb{C}$ there exist exactly $m-1$ solutions of the
equation $A_s(z)=x $ (assume that $b^{(m-2)}(s) \neq 0$) which we
denote by $z_j(x,s)$, $j=1,\ldots,m-1$. We label the solutions by
their absolute value so that
\begin{equation}
\label{absolute_value_solutions}
    |z_{1}(x,s)| \geq |z_2(x,s)| \geq \cdots \geq |z_{m-1}(x,s)| > 0.
\end{equation}
We put
\begin{equation} \label{Gamma1s}
    \Gamma_1(s) = \{ x\in \mathbb{C} \mid |z_1(x,s)|=|z_2(x,s)| \},
    \end{equation}
which is a finite union of analytic arcs.

We will use the following classical theorem on the behavior of the eigenvalues
of $T_n(A_s)$ as $n \to \infty$.
\begin{thm} \label{Toeplitztheorem}
As $n \to \infty$ the eigenvalues of $T_n(A_s)$ accumulate on the contour
\eqref{Gamma1s}.

The sequence of normalized counting measures of the eigenvalues of $T_n(A_s)$
converges weakly as $n\to \infty$ to the Borel probability measure $\mu^s_1$
on $\Gamma_1(s)$, given by
\begin{equation} \label{measure_mu_for_recurrence}
    d\mu^s_1(x) = \frac{1}{2\pi i}
    \left( \frac{z'_{1-}(x,s)}{z_{1-}(x,s)}-\frac{z'_{1+}(x,s)}{z_{1+}(x,s)} \right)dx,
\end{equation}
where $'$ denotes the derivative with respect to $x$. In \eqref{measure_mu_for_recurrence},
we have that $dx$ is the complex line element on $\Gamma_1(s)$ and
$z_{1\pm}(x,s)$ is the limiting value of $z_1(\tilde{x},s)$ as $\tilde{x}\to x$
from the $\pm$ side on each of the arcs in $\Gamma_1(s)$ (the complex line element induces an orientation
on $\Gamma_1(s)$ and the $+$ side ($-$ side) is on the left (right) as one traverses
$\Gamma_1(s)$ according to the orientaton).
\end{thm}
\begin{proof}
The fact that the eigenvalues accumulate on $\Gamma_1(s)$
was shown by Schmidt and Spitzer \cite{SS}. The result about
the limit of the normalized eigenvalue counting measures is due
to Hirschman \cite{Hirschman}, see also \cite[Chapter~11]{BG}.
The precise form \eqref{measure_mu_for_recurrence} for the limiting
measure was given in \cite{DK1}.
\end{proof}
For later use we also note that by \cite[Proposition~4.2]{DK1},
\begin{equation} \label{mu_1_identity_1}
    \int \frac{d \mu^s_1(y)}{x-y} = \frac{z'_1(x,s)}{z_1(x,s)},
        \qquad \text{for } x \in \mathbb{C}\setminus \Gamma_1(s),
\end{equation}
which also characterizes the measure $\mu^s_1$.

Our first result states that, under certain conditions, the
polynomials $P_{k,n}$ satisfying the recurrence \eqref{recurrence_m}
have a limiting zero distribution, which is an average of
the measures \eqref{measure_mu_for_recurrence}. The average
is with respect to the parameter $s$.

Theorem~\ref{Theorem_zero_distribution_general} is an extension
of Theorem~3.1 of \cite{CCVA} to polynomials
satisfying an $m$-term recurrence relation instead of a specific four-term recurrence
relation as in \cite{CCVA}. The analogous result for orthogonal polynomials
satisfying a three-term recurrence is from \cite{KVA}. See also
\cite{KS} and \cite{ShapT} for related results.

\begin{thm} \label{Theorem_zero_distribution_general}
Let for each $n\in \mathbb{N}$, $m-1$ sequences $\{b_{k,n}^{(j)}\}_{k=0}^\infty$,
$j=0,\ldots,m-2$, of real coefficients be given and assume that there exist
continuous functions $b^{(j)}:[0,\infty)\to \mathbb{R}$, $j=0,\ldots,m-2$,
such that for each $s\geq 0$,
\begin{equation}
\label{Convergence_coefficients_m}
\lim_{k/n\to s} b^{(j)}_{k,n}=b^{(j)}(s),\quad j=0,\ldots,m-2.
\end{equation}
Let $P_{k,n}$ be the monic polynomials generated by the recurrence \eqref{recurrence_m}
and suppose that
\begin{itemize}
\item[\rm (a)]  the polynomials $P_{k,n}$ have real and simple zeros $x_1^{k,n} < \cdots < x_k^{k,n}$
satisfying for each $k$ and $n$ the interlacing property
\[ x_j^{k+1,n}< x_j^{k,n}<x_{j+1}^{k+1,n}, \qquad \text{for } j=1,\ldots,k, \]
\item[\rm (b)] $\Gamma_1(s) \subset \mathbb R$ for every $s > 0$, where $\Gamma_1(s)$
is given by \eqref{Gamma1s}.
\end{itemize}
Then the normalized zero counting measures $\nu(P_{k,n}) = \frac{1}{k} \sum_{j=1}^k \delta_{x_j^{k,n}}$
have a weak limit as $k, n \to \infty$ with $k/n \to \xi > 0$ given by
\begin{equation} \label{limitingmeasure}
    \lim_{k/n\to \xi} \nu(P_{k,n})= \frac{1}{\xi}\int_0^\xi \mu^{s}_1 \, ds
    \end{equation}
where $\mu_1^s$ is the measure \eqref{measure_mu_for_recurrence}.
\end{thm}
The proof of Theorem~\ref{Theorem_zero_distribution_general} is
given in Section~\ref{Section_zero_distribution}.

\subsection{Multiple orthogonal polynomials associated with modified Bessel functions}

We want to apply Theorem~\ref{Theorem_zero_distribution_general} to
the multiple orthogonal polynomials $B_k$ associated with the modified
Bessel function.

\subsubsection{Interlacing} The assumption (a) of Theorem~\ref{Theorem_zero_distribution_general}
will be satisfied since we have the following.

\begin{prop}
\label{Interlacing_Bk} Let $a > 0$, $\alpha > -1$, $0 < t < T$.
Then the polynomials $B_k$ generated by the recurrence
\eqref{eq:polynomialsBk} with recurrence coefficients
\eqref{eq:coefficientsbcdk} have real and simple zeros in
$(0,\infty)$ with the interlacing property.
\end{prop}

The proof of Proposition~\ref{Interlacing_Bk} is given in
Section~\ref{interlacing}.

\subsubsection{First rescaling}
In the first rescaling we let $t$ and $T$
depend on $n$, while $\alpha$ and $a$ remains fixed. We replace
\[ t \mapsto t/(2n), \qquad T \mapsto 1/(2n) \]
and we obtain the recurrence coefficients as in
\eqref{eq:coefficientsbcdkn}. The scaling limits of the recurrence
coefficients indeed exist:
\begin{equation}
\label{scaling_limits}
\begin{split}
\lim_{k/n\to s} b_{k,n} &= b(s) = a(1-t)^2+2st(1-t),\\
\lim_{k/n\to s} c_{k,n} &= c(s) = 2ast(1-t)^3+s^2t^2(1-t)^2,\\
\lim_{k/n\to s} d_{k,n} &= d(s) = as^2t^2(1-t)^4.
\end{split}
\end{equation}

Then as in \eqref{sympbolsAs} we have the associated family of symbols
\begin{align} \label{family_of_symbols0}
    A_s(z) & =z+b(s)+c(s)z^{-1}+d(s)z^{-2}
\end{align}
and the solutions $z_1(x,s)$, $z_2(x,s)$ and $z_3(x,s)$ of the algebraic equation $A_s(z)= x$.
We define $\Gamma_1(s)$ as in \eqref{Gamma1s} and similarly
\begin{equation} \label{Gamma2s}
    \Gamma_2(s) = \{ x \in \mathbb{C} \mid |z_2(x,s)| = |z_3(x,s)| \}.
    \end{equation}

The symbol \eqref{family_of_symbols0} with the functions $b(s)$, $c(s)$
and $d(s)$ from \eqref{scaling_limits} allows for a factorization
\begin{align} \label{family_of_symbols}
    A_s(z) = \frac{(z+a(1-t)^2)(z+st(1-t))^2}{z^2}.
\end{align}
From \eqref{family_of_symbols} we see that $A_s$ has three negative zeros,
namely a double zero at $-st(1-t)$ and a simple zero at $-a(1-t)^2$.
For the special value
\[ s = s^* = \frac{a(1-t)}{t} \]
the three zeros of the symbol coincide.

These facts are used to prove the following.
\begin{prop}
\label{Proposition_Gamma_1_2}
For each $s > 0$, we have that $\Gamma_1(s) \subset [0,\infty)$
and $\Gamma_2(s) \subset (-\infty,0]$. More precisely, there exist
$\eta(s) \leq 0 \leq \beta(s) < \gamma(s) $
so that
\begin{align} \label{Gamma12s-in-first-scaling}
    \Gamma_1(s) = [\beta(s), \gamma(s)], \qquad
    \Gamma_2(s) = (-\infty, \eta(s)].
    \end{align}
In addition we have
\begin{enumerate}
\item[\rm (a)] $s \mapsto \gamma(s)$ is strictly increasing for $s > 0$,  $\lim_{s\rightarrow 0^+}
\gamma(s)=a(1-t)^2$, and $\lim_{s\rightarrow \infty} \gamma(s)=+\infty$,
\item[\rm (b)] $s \mapsto \beta(s)$ is positive and strictly decreasing
for $0 < s < s^* = \frac{a(1-t)}{t} $ and $\beta(s) = 0$ for $s \geq s^*$.
Furthermore, $\lim_{s\rightarrow 0^+} \beta(s)=a(1-t)^2$,
\item[\rm (c)] $\eta(s) = 0$ for $0 < s \leq s^*$ and $s \mapsto \eta(s)$ is
negative, strictly decreasing for $s > s^*$ and $\lim_{s\rightarrow \infty} \eta(s)=-\infty$.
\end{enumerate}
\end{prop}
From the proposition it follows that the sets $\Gamma_1(s)$ and
$\Gamma_2(s)$ are intervals, and that $\Gamma_1(s)$ is increasing
as $s$ increases, while $\Gamma_2(s)$ decreases. See the left part
of Figure~\ref{figure_curves_gamma} for an illustrative plot of
$\gamma(s)$, $\beta(s)$ and $\eta(s)$, as a function of $s$.

\begin{figure}
\centering
\begin{overpic}[width=12cm,height=6.07cm]{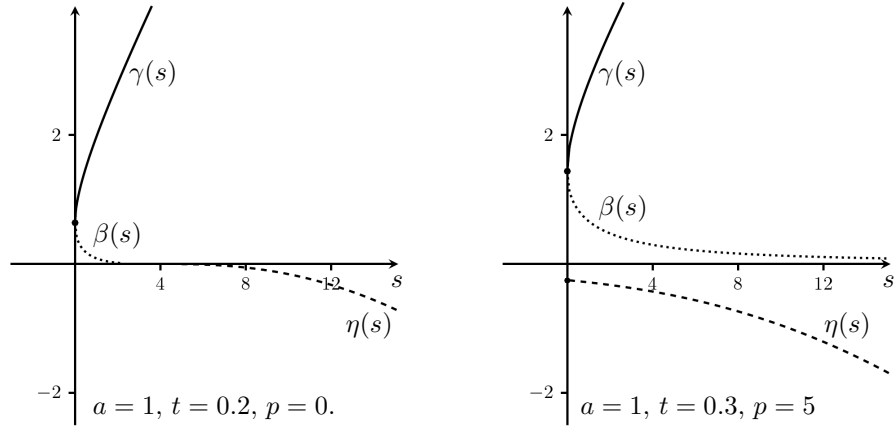}
\put(43,17){$s$}
\put(97.5,17){$s$}

\put(38,12){$\eta(s)$}
\put(10,22){$\beta(s)$}
\put(14,40){$\gamma(s)$}
\put(10,3){$a=1$, $t=0.2$, $p=0.$}

\put(66,25){$\beta(s)$}
\put(66,40){$\gamma(s)$}
\put(91,12){$\eta(s)$}
\put(64,3){$a=1$, $t=0.3$, $p=5$}
\end{overpic}
\caption{\label{figure_curves_gamma}Graph of the curves $\beta(s)$, $\gamma(s)$ and $\eta(s)$ for the first and second rescaling.}
\end{figure}

As a result of Propositions~\ref{Interlacing_Bk} and \ref{Proposition_Gamma_1_2} we see that the assumptions (a) and (b)
of Theorem~\ref{Theorem_zero_distribution_general} are satisfied, and
so for each $\xi > 0$, the weak limit of the normalized zero counting measures
of the polynomials $B_{k,n}$ as $k,n\to \infty$, $k/n \to \xi$, exists
and is given by \eqref{limitingmeasure}.

\subsubsection{Second rescaling}
In the second rescaling we let the parameter $\alpha$ increase as $n$ increases.
For this we change the variables $t$, $T$ and $\alpha$ by
\[ t\mapsto t/(2n),\qquad T \mapsto 1/(2n), \qquad \alpha \mapsto p n \]
with $p > 0$.

The recurrence coefficients \eqref{eq:coefficientsbcdk} now have
scaling limits that also depend on $p$. Indeed,
\begin{equation}
\label{scaling_limits2}
\begin{split}
b(s)&=a(1-t)^2+2st(1-t)+t(1-t)p,\\
c(s)&=2ast(1-t)^3+s^2t^2(1-t)^2+st^2(1-t)^2p,\\
d(s)&=as^2t^2(1-t)^4.
\end{split}
\end{equation}
Now the symbol \eqref{family_of_symbols} depends on  $p$
and we can again explicitly factorize
\begin{equation}
\label{family_of_symbols2}
    A_s(z) = \frac{(z+st(1-t))(z^2+(1-t)(a(1-t)+(s+p) t)z + ast(1-t)^3)}{z^2}.
    \end{equation}
There are again three negative zeros of the symbol, but now all three
zeros are simple.

We again have $\Gamma_1(s)$ and $\Gamma_2(s)$ as in \eqref{Gamma1s} and \eqref{Gamma2s},
and we prove the following.
\begin{prop}
\label{Proposition_Gamma_1_2_alpha}
Let $p > 0$.
Then for each $s > 0$, we have that $\Gamma_1(s) \subset (0,\infty)$
and $\Gamma_2(s) \subset (-\infty,-\frac{p^2 t^2}{4 a}]$. More precisely, there exist
$\eta(s) < 0 < \beta(s) < \gamma(s) $ so that
\begin{align} \label{Gamma12s-in-second-scaling}
    \Gamma_1(s) = [\beta(s), \gamma(s)], \qquad
    \Gamma_2(s) = (-\infty, \eta(s)].
    \end{align}
In addition we have
\begin{enumerate}
\item[\rm (a)] $s \mapsto \gamma(s)$ is strictly increasing for $s > 0$,
$\lim_{s\to 0+} \gamma(s)=(1-t)(a(1-t)+p t)$  and
$\lim_{s\to \infty} \gamma(s)=\infty$,
\item[\rm (b)] $s \mapsto \beta(s)$ is positive and strictly decreasing for $s > 0$ with
$\lim_{s\to 0+} \beta(s)=(1-t)(a(1-t)+p t)$ and
$\lim_{s\to \infty} \beta(s)=0$,
\item[\rm (c)] $s \mapsto \eta(s)$ is
negative and strictly decreasing for $s > 0$ with
$\lim_{s\to 0+} \gamma(s)=-\frac{p^2 t^2}{4 a}$ and
$\lim_{s\to \infty} \eta(s)=-\infty$.
\end{enumerate}
\end{prop}

See the right part of Figure~\ref{figure_curves_gamma} for a plot
of the functions $\beta(s)$, $\gamma(s)$ and $\eta(s)$ in the case
$p = 5$.

Also in the second scaling  the assumptions of Theorem~\ref{Theorem_zero_distribution_general}
are satisfied, and again it follows that for each $\xi > 0$, the weak limit
of the normalized zero counting measures
of the polynomials $B_{k,n}$ as $k,n\to \infty$, $k/n \to \xi$, exists
and is given by \eqref{limitingmeasure}.

The proof of Propositions~\ref{Proposition_Gamma_1_2} and \ref{Proposition_Gamma_1_2_alpha}
are given in Section~\ref{sec:proof_propositions}.

\subsection{Equilibrium problem}

In both scalings we find for each $\xi > 0$ a probability measure of the form
\begin{align} \label{nu1xi-as-integral}
    \nu_1^{\xi} = \frac{1}{\xi} \int_0^{\xi} \mu_1^s \, ds
    \end{align}
as the weak limit of the normalized zero counting measures.
The main result of the paper is that this measure can also be obtained
as the first component of a vector of measures $(\nu_1^{\xi}, \nu_2^{\xi})$
that satisfies a vector equilibrium problem. For the case $p = 0$
and $\xi = 1$, it reduces to the vector equilibrium problem
\eqref{eq:energyfunctional}--\eqref{eq:constraintdensity}
stated in the introduction.

We will use a recent result of Duits and Kuijlaars \cite{DK1}
which in the present context says that the measure $\mu_1^s$
that gives the limiting eigenvalue distribution of the
Toeplitz matrices with symbol \eqref{family_of_symbols} is
part of a vector $(\mu_1^s, \mu_2^s)$ that is characterized
by a vector equilibrium problem.

The second measure $\mu_2^s$ is supported on $\Gamma_2(s)$
(see \eqref{Gamma2s}) and is given by
\begin{equation} \label{eq:densitymu2}
    d\mu^s_2(x) = \frac{1}{2\pi i}
    \left(\frac{z'_{2-}(x,s)}{z_{2-}(x,s)}-\frac{z'_{2+}(x,s)}{z_{2+}(x,s)} \right)dx
    \end{equation}
for $x \in \Gamma_2(s)$. Then $\mu_2^s$ is indeed  a positive measure
on $\Gamma_2(s)$ with total mass $1/2$.

The result of \cite{DK1} in this special case is the following.
\begin{thm} \label{equilibrium_problem_Toeplitz}
For each $s > 0$ we have that the vector $(\mu_1^s, \mu_2^s)$
is the unique minimizer for the energy functional
\begin{multline}
\label{equilibrium_problem_mu1_mu2}
    \iint \log \frac{1}{|x-y|} d\mu_1(x)d\mu_1(y) +  \iint \log \frac{1}{|x-y|} d\mu_2(x)d\mu_2(y)\\
    - \iint \log \frac{1}{|x-y|}d\mu_1(x)d\mu_2(y)
\end{multline}
among all vectors $(\mu_1, \mu_2)$ satisfying $\supp(\mu_j) \subset \Gamma_j(s)$ for $j=1,2$, and
\[ \int d \mu_1 = 1, \qquad \int d\mu_2 = \frac{1}{2}. \]

The measures $\mu_1^s$ and $\mu_2^s$ satisfy for some constant $\ell^s$,
\begin{align}
\label{variational_mu_1}
2\int \log |x-y|d\mu^s_1(y)-\int \log |x-y|d\mu^s_{2}(y) & =\ell^s, \qquad x\in \Gamma_1(s), \\
\label{variational_mu_2} 2\int \log |x-y|d\mu^s_2(y)-\int \log
|x-y|d\mu^s_{1}(y) & =0, \qquad \, x\in \Gamma_2(s).
\end{align}
\end{thm}
The conditions \eqref{variational_mu_1}--\eqref{variational_mu_2}
are the Euler-Lagrange variational conditions for the vector equilibrium
problem.

Recall that we have \eqref{nu1xi-as-integral} and similarly we put for $\xi > 0$,
\begin{align} \label{nu2xi-as-integral}
    \nu_2^{\xi} = \frac{1}{\xi} \int_0^{\xi} \mu_2^s \, ds.
    \end{align}
Then $\nu_2^{\xi}$ is a measure on $\bigcup_{s<\xi}\Gamma_2(s)=(-\infty,-p^2t^2/4a]$
with total mass $1/2$. We obtain the vector equilibrium problem for $(\nu_1^{\xi}, \nu_2^{\xi})$ by
integrating the vector equilibrium problem for $(\mu_1^s, \mu_2^s)$ with
respect to $s$, in
particular the variational conditions \eqref{variational_mu_1}--\eqref{variational_mu_2}.
A complication is that the intervals $\Gamma_1(s)$ and $\Gamma_2(s)$ are varying with $s$.
The fact that $\Gamma_1(s)$ is increasing with $s$ induces, after integration,
an external field on $\nu_1^{\xi}$. The fact that $\Gamma_2(s)$ is decreasing
as $s$ increases leads to the upper constraint on $\nu_2^{\xi}$.

\begin{figure}
\centering
\begin{overpic}[width=12.45cm,height=5.72cm]{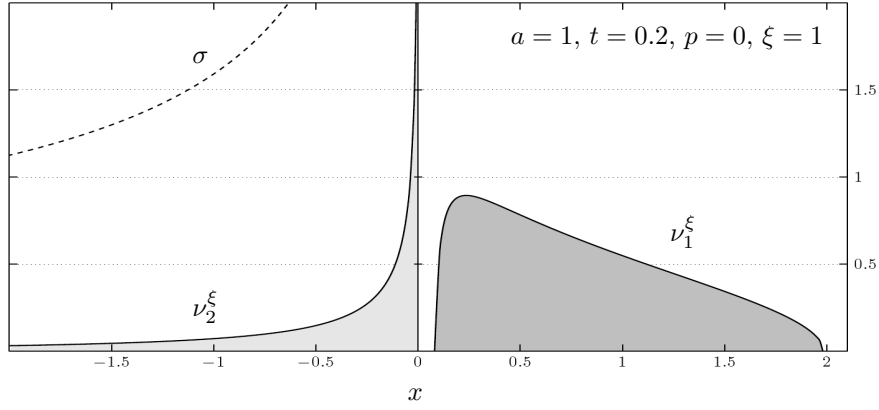}
    \put(75,17){$\nu_1^\xi$}
    \put(24,36){$\sigma$}
    \put(24,9){$\nu_2^\xi$}
    \put(58,38){$a=1$, $t=0.2$, $p=0$, $\xi=1$}
    \put(47,0) {$x$}
\end{overpic}
\caption{Graphs of the densities of $\sigma$ (dashed) and
$\nu_2^\xi$ on the negative half line and the density of
$\nu_1^\xi$ on the positive half line, for the case $p = 0$
and $t < t^*$.}\label{plot:densities_sigma_nu_a}
\end{figure}

\begin{thm} \label{proposition_V(x)_nu_1}
Define
\begin{align} \label{definition-of-V}
    V(x) = \int_0^{\infty} \log \left|\frac{z_1(x,s)}{z_2(x,s)} \right| \, ds
    \end{align}
    and
\begin{align} \label{definition-of-sigma}
    \sigma = \int_0^{\infty} \mu_2^s \, ds.
    \end{align}

Then for every $\xi > 0$, the vector of measures $(\nu^\xi_1,\nu^\xi_2)$ is the
unique minimizer for the energy functional
\begin{multline} \label{functionalwithV}
    \iint \log \frac{1}{|x-y|} d\nu_1(x)d\nu_1(y) + \iint \log \frac{1}{|x-y|}d\nu_2(x)d\nu_2(y)\\
    -\iint \log \frac{1}{|x-y|} d\nu_1(x)d\nu_2(y) +\frac{1}{\xi} \int V(x) d\nu_1(x),
\end{multline}
over all vectors of measures $(\nu_1,\nu_2)$ such that $\supp(\nu_1)\subset [0,\infty)$,
$\int d\nu_1 = 1$, $\supp(\nu_2)\subset (-\infty,0]$, $\int d\nu_2 = 1/2$, and
\[ \nu_2 \leq \frac{1}{\xi} \sigma. \]
\end{thm}

The measures $\nu_1^{\xi}$ and $\nu_2^{\xi}$ are characterized by
the following variational conditions
\begin{align}
\label{variational_nu_1}
2\int \log |x-y| d\nu^\xi_1(y) -\int \log |x-y| d\nu^\xi_2(y) + \frac1\xi V(x)
    &\begin{cases} = \ell, \quad \text{for } x\in  \supp (\nu^\xi_1), \\
    \leq \ell, \quad \text{for } x\in [0,\infty), \end{cases}
    \end{align}
for some $\ell$, and
    \begin{align}
\label{variational_nu_2}
2\int \log |x-y| d\nu^\xi_2(y) -  \int \log |x-y| d\nu^\xi_1(y)  &
    \begin{cases} =0, \quad \text{for } x \in \operatorname{supp}(\sigma-\xi\nu^\xi_2), \\
    > 0,\quad \text{for } x \in \mathbb{C} \setminus \operatorname{supp}(\sigma-\xi\nu^\xi_2).
    \end{cases}
\end{align}

The proof of Theorem~\ref{proposition_V(x)_nu_1}
is given in Section~\ref{sec:proof_theorem_Var_conditions}.

\subsection{Evaluation of $V$ and $\sigma$}
In a final result we are able to evaluate the integrals \eqref{definition-of-V}
and \eqref{definition-of-sigma} that define $V$ and $\sigma$.

\begin{thm} \label{thm:external_field}
For every $p\geq 0$, we have
\begin{multline}
\label{external_field}
    V(x) = \frac{x}{t(1-t)} - \frac{\sqrt{p^2t^2 + 4ax}}{t}
        - p \log \left( \sqrt{p^2 t^2 + 4ax} - p t \right) \\
        + \frac{a(1-t)}{t} + p \log \left(2a(1-t) \right),
\end{multline}
and $\sigma$ is the measure on $(-\infty,0]$ with density
\begin{equation}
\label{eq:density_constraint} \frac{d\sigma(x)}{dx} = \left\{ \begin{array}{cl}
    \frac{\sqrt{4a|x|-p^2 t^2}}{2\pi t |x|}, &  \text{ for }
        x\in \left(-\infty,- \frac{p^2t^2}{4a}\right], \\
    0 &  \text{ for } x \in \left(- \frac{p^2t^2}{4a},0\right].
    \end{array} \right.
\end{equation}
\end{thm}

\begin{figure}
\centering
\begin{overpic}[width=8.26cm,height=5.72cm]{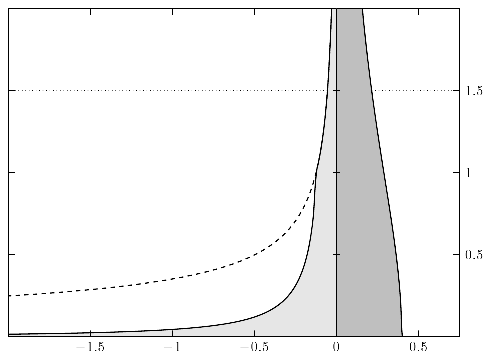}
    \put(80,30){$\nu_1^\xi$}
    \put(40,20){$\sigma$}
    \put(40,12){$\nu_2^\xi$}

    \put(10,57){$a=1$, $t=0.9$, $p=0$, $\xi=1$}
    \put(50,0) {$x$}
\end{overpic}
\caption{Graphs of the densities of $\sigma$ (dashed) and
$\nu_2^\xi$ on the negative half line and the density of
$\nu_1^\xi$ on the positive half line, for the case $p = 0$
and $t>t^*$.}\label{plot:densities_sigma_nu_b}
\end{figure}

Note that for $p = 0$, the external field \eqref{external_field} is
\begin{equation}
    V(x) = \frac{x}{t(1-t)}-\frac{2\sqrt{ax}}{t} + \frac{a(1-t)}{t}=\frac{1}{t(1-t)}(\sqrt{x}-\sqrt{a}(1-t))^2
\end{equation}
and the constraint \eqref{eq:density_constraint} is the measure
with density
\[ \frac{d\sigma(x)}{dx}=\frac{\sqrt{a}}{\pi t}|x|^{-1/2},\quad x\in(-\infty,0]. \]

% \begin{rem}
% The measures $\nu^{\xi}_1$ and $\nu_2^{\xi}$, for $\xi=1$, correspond exactly to the
% measures $\nu_1$ and $\nu_2$ given in \cite[Section 11]{KMW}. The measure $\nu^1_1=\nu_1$
% gives the limiting mean density of the positions of the paths at the time $t$ as $n\to \infty$.
% If we move $t$ from $0$ to $1$ then the support of $\nu_1$ fill out a domain in the $tx$-plane
% whose boundary is plotted with a line in Figure 1.
% \end{rem}

\begin{rem}
Figure~\ref{plot:densities_sigma_nu_a} shows the graph of the
densities of $\sigma$, $\nu_1^\xi$ and $\nu_2^\xi$ in the case $p=0$ and $t$
below the critical time $t^*$. This corresponds to the case
where the non-intersecting paths did not come to the hard edge.
The constraint $\sigma$ is not active
and $\nu_1^\xi$ is supported on a interval which is at a positive distance from zero.
Figure~\ref{plot:densities_sigma_nu_b} illustrates the case $p=0$ for $t>t^*$. Here the constraint $\sigma$ is active in an interval.

Figure~\ref{plot:densities_sigma_nu_p} shows the densities of $\sigma$, $\nu_1^\xi$ and $\nu_2^\xi$
for $p>0$. For all values of $t$, the constraint $\sigma$ is active in some interval
and $\nu_1^\xi$ is supported on an interval which is at a positive distance from zero.
\end{rem}

\begin{rem}
It was pointed out that if we let $n\to \infty$, the paths fill out a region in the $tx$-plane.
This can be observed in the left figure of  Figure~\ref{figureNon} for the case $p=0$ and
in the right figure of Figure~\ref{figureNon} for the case $p>0$.
The region filled by the paths,
for a fixed time $t\in [0,1]$, is exactly the interval $[\beta(1),\gamma(1)]$.
We can obtain $\beta(1)$ and $\gamma(1)$ by computing the zeros of the discriminant
of the polynomial $z^2A_1(z)-z^2x$ with respect to $z$, which is
the  following algebraic equation of degree three in $x$:
\begin{multline*}
    4ax^3 - \big[8a^2(1-t)^2 + 4at(1-t)(2p+5) -t^2(p+1)^2 \big] x^2 \\
    + (1-t) \big[4a^3(1-t)^3 + 4a^2 t (1-t)^2 (2p-3) + 2at^2 (1-t)(p^2+p+6)
    -2t^3(p+2)(p+1)^2\big]  x \\
    + p^2 t^2 (1-t)^2 \big[ a^2 (1-t)^2 + 2a t(1-t)(p-1)+ t^2(p+1)^2 \big]
    =0.
\end{multline*}
If $p=0$ then the algebraic equation reduces to
\begin{equation}
x(4ax^2-x(8a^2(1-t)^2+20at(1-t)-t^2)+4(1-t)(a(1-t)-t)^3)=0.
\end{equation}
We can compute explicit expressions for the solutions in this case
\begin{align*}
x_1(t)&=0,\\
x_2(t)&=\frac{1}{8a}\big(8a^2(1-t)^2-t(t-20a(1-t))-\sqrt{t(t+8a(1-t))^3}\big),\\
x_3(t)&=\frac{1}{8a}\big(8a^2(1-t)^2-t(t+20a(1-t))+\sqrt{t(t+8a(1-t))^3}\big).
\end{align*}
\end{rem}
\begin{figure}
\centering
\begin{overpic}[width=10cm,height=5.72cm]{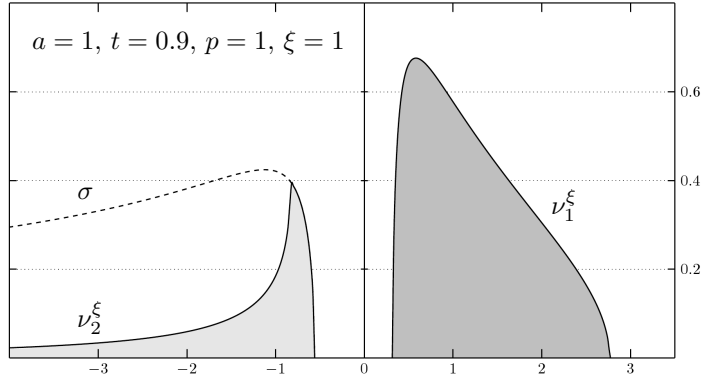}
    \put(14,27){$\sigma$}
    \put(77,25){$\nu_1^\xi$}
     \put(14,10){$\nu_2^\xi$}
     \put(8,47){$a=1$, $t=0.9$, $p=1$,
$\xi=1$}
\end{overpic}
\caption{Graphs of the densities of $\sigma$ (dashed) and
$\nu_2^\xi$ on the negative real line and the density of
$\nu_1^\xi$ on the positive real line for $p > 0$.}\label{plot:densities_sigma_nu_p}
\end{figure}

\begin{rem} If we let $a\rightarrow 0$ and $s=1$, then the symbol \eqref{family_of_symbols2} becomes
$$\frac{(1+zt(1-t))^2}{z^2}+\frac{p t(1-t)(1+zt(1-t))}{z}.$$
The solution $z_3(x,1)$ tends to zero as $a\rightarrow 0$. On the
other hand, the solutions $z_1(x,1)$ and $z_2(x,1)$ have limits
(with appropriate choice of $\pm$-sign)
\begin{align*}
z_1(x)&=-\frac{1}{2}t(1-t)(p+2)+\frac{x}{2} +\frac12 \left((x-\rho_1(t))(x-\rho_2(t))\right)^{1/2},\\
z_2(x)&=-\frac{1}{2}t(1-t)(p+2)+\frac{x}{2} -\frac12
\left((x-\rho_1(t))(x-\rho_2(t))\right)^{1/2},
\end{align*}
as $a \to 0$, where
$$\rho_1(t)=t(1-t)(p+2-2\sqrt{p+1}),
\qquad \rho_2(t)=t(1-t)(p+2+2\sqrt{p+1}).$$
One can show that $\nu_1$ is the
Marchenko-Pastur distribution, see e.g.\ \cite{HP}, with density
$$\frac{d\nu_1(x)}{dx}=\frac{\sqrt{(\rho_2(t)-x)(x-\rho_1(t))}}{2\pi t(1-t)x},
\quad x\in[\rho_1(t),\rho_2(t)].$$
Figure~\ref{non_intersecting_March_Pastur} shows simulations of
$50$ non-intersecting squared Bessel paths for the case $a=0$ and
the boundaries $\rho_1(t)$ and $\rho_2(t)$ of the region filled by
the paths as $n\rightarrow \infty$.
\end{rem}

\begin{figure}
\centering
\begin{overpic}[width=5.9cm,height=4.72cm]{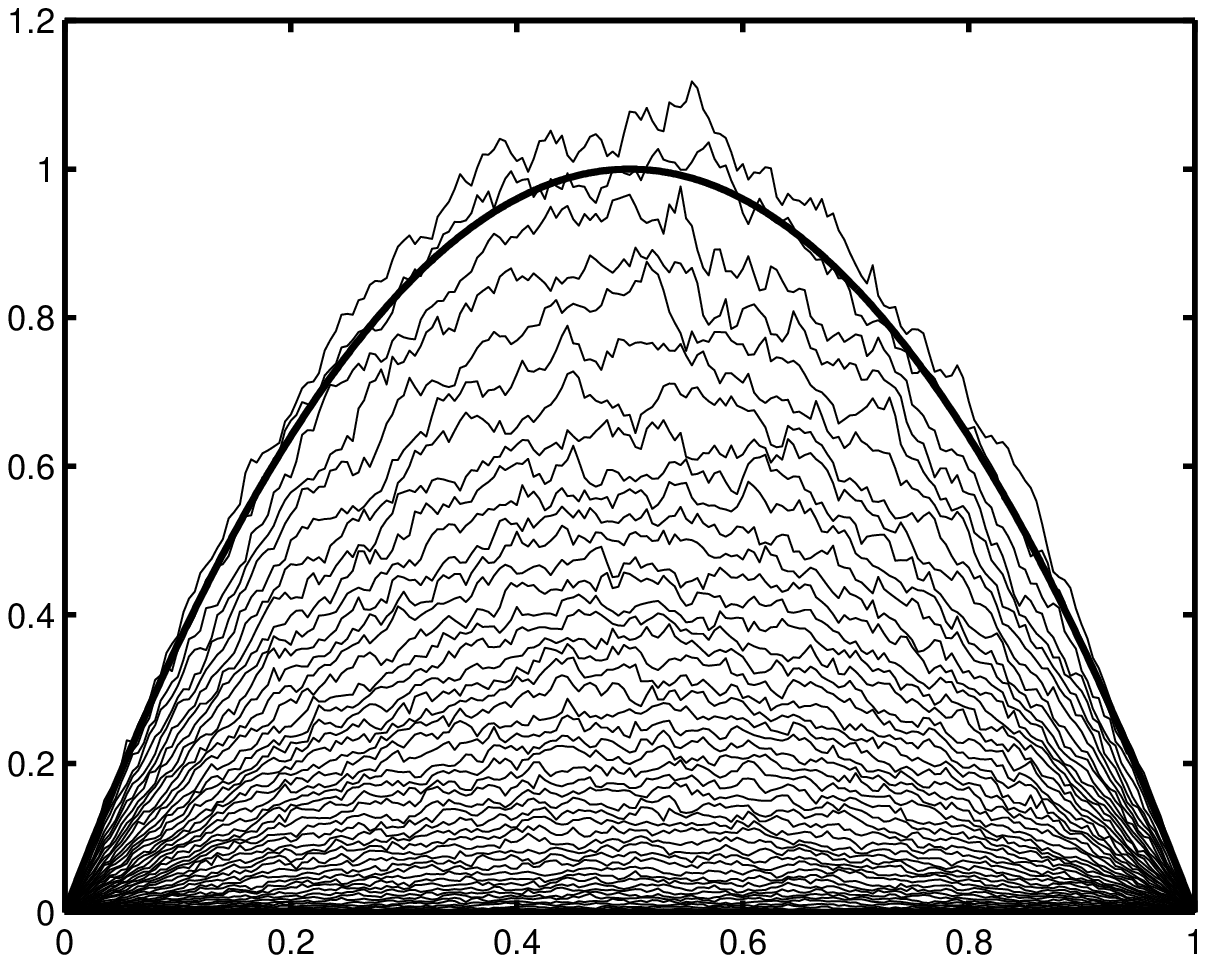}
\put(58,67.5){$a=0$, $p = 0$}
\put(50,2){$t$}
\put(5,40){$x$}
\end{overpic}
\hspace{.5cm}
\begin{overpic}[width=5.9cm,height=4.72cm]{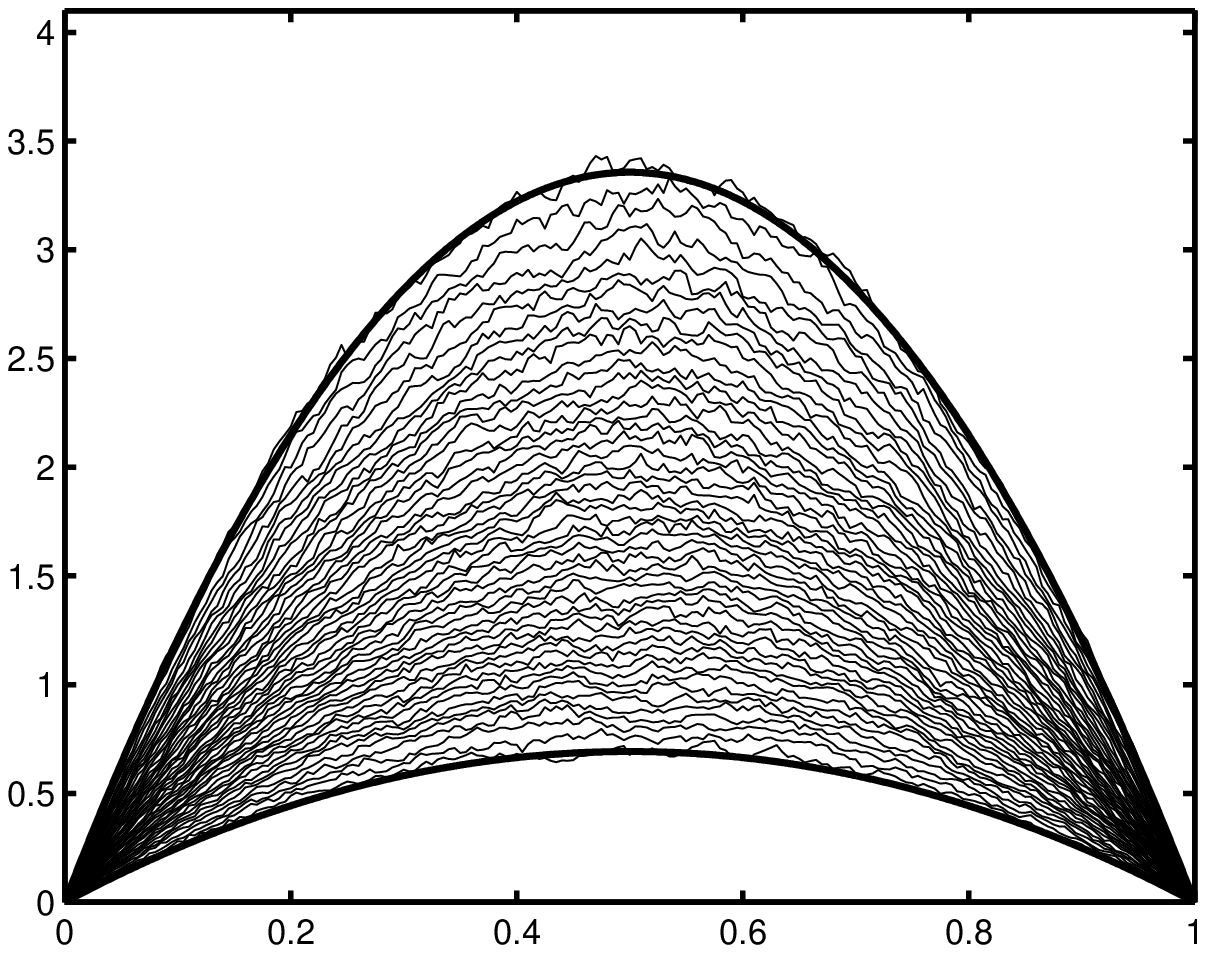}
\put(57,67.5){$a=0$, $p = 5$}
\put(50,2){$t$}
\put(5,40){$x$}
\end{overpic}
\caption{
\label{non_intersecting_March_Pastur}
Numerical simulation of $50$ non-intersecting paths for $a=0$
and $p=0$ (left), $p>0$ (right).}
\end{figure}

The rest of the paper is devoted to the proofs of the theorems and
propositions stated above. Theorem~\ref{Theorem_zero_distribution_general}
is proved in Section~\ref{Section_zero_distribution}, Proposition~\ref{Interlacing_Bk}
in Section~\ref{interlacing}, Propositions~\ref{Proposition_Gamma_1_2}
and \ref{Proposition_Gamma_1_2_alpha}
in Section~\ref{sec:proof_propositions},
Theorem~\ref{proposition_V(x)_nu_1} in Section~\ref{sec:proof_theorem_Var_conditions},
and finally Theorem~\ref{thm:external_field} is proved in
Section~\ref{sec:externalfield}.

\section{Proof of Theorem~\ref{Theorem_zero_distribution_general}}
\label{Section_zero_distribution}

Before proving Theorem~\ref{Theorem_zero_distribution_general} we will state a result
concerning ratio asymptotics of polynomials satisfying a recurrence relation \eqref{recurrence_m}
with varying recurrence coefficients.

The following lemma will be used in the proofs of  Theorem~\ref{Theorem_zero_distribution_general}
and Lemma~\ref{theorem_before_zeros_polynomials}.
\begin{lem}
\label{lemma_polynomials_interlacing}
Suppose that the zeros of the monic polynomials $P_k$ and $P_{k+1}$, with degrees $k$ and $k+1$,
respectively, are simple and real, lie in an interval $[-R,R]$ for some $R>0$, and are interlacing.
Then we have
\begin{align}
    \label{inequality1}
    & \left| \frac{P_{k}(z)}{P_{k+1}(z)}\right|\leq \frac{1}{\dist(z,[-R,R])} \qquad \text{for } z\in \mathbb{C}\setminus [-R,R], \\
    \label{inequality2}
    & \left| \frac{P_{k}(z)}{P_{k+1}(z)}\right|\geq \frac{1}{2|z|} \qquad \text{for }  |z|>R, \\
    \label{inequality3}
    & \left| \left( \frac{P_{k}(z)}{P_{k+1}(z)} \right)'\right| \leq \frac{1}{\dist(z,[-R,R])^2} \qquad
        \text{for } z\in \mathbb{C}\setminus [-R,R].
\end{align}
\end{lem}
\begin{proof}
Inequalities \eqref{inequality1} and \eqref{inequality2} can be found in \cite[Lemma 2.2]{KVA}.
The proof of \eqref{inequality3} is similar to that of \eqref{inequality1}.
\end{proof}

We consider the doubly indexed sequences of polynomials $\{P_{k,n}\}$ generated
by the $m$-term recurrence relation \eqref{recurrence_m} and
we assume that the recurrence coefficients have scaling limits
\[ \lim_{k/n\to s} b_{k,n}^{(j)}=b^{(j)}(s). \]

\begin{lem} \label{theorem_before_zeros_polynomials}
Under the assumptions of Theorem~\ref{Theorem_zero_distribution_general}
we have that for each $s > 0$, there exists $R > 0$ so that
all zeros of $P_{k,n}$ belong to $[-R,R]$ whenever $k\leq (s+1)n$.
Moreover,
\begin{equation}
\label{asymptotic_dist_of_zeros}
\lim_{k/n\to s} \frac{P_{k+1,n}(x)}{P_{k,n}(x)} = z_1(x,s),
\end{equation}
uniformly on compact subsets of $\mathbb{C} \setminus [-R,R]$.
\end{lem}
\begin{proof}
Fix $s >0$.  The convergence \eqref{Convergence_coefficients_m}
imply that the recurrence coefficients are uniformly bounded if $k/n$ is restricted
to compact subsets of $[0,\infty)$. So, the number $R$ defined by
\begin{equation} \label{definitionR}
    R:= \sup \{1+|b^{(0)}_{k,n}|+|b^{(1)}_{k,n}|+\cdots+|b^{(m-2)}_{k,n}| \mid k\leq (s+1) n\}
    \end{equation}
is finite.

By the recurrence \eqref{recurrence_m} we have $P_{k,n}(x)=\det(zI_k-M_{k,n})$
where $M_{k,n}$ is the matrix
\begin{equation}
\label{Matrix_M_k}
M_{k,n}=\begin{pmatrix}
          b^{(0)}_{0,n}     &  1            & 0      &                    &           &          &  \\
          b^{(1)}_{1,n}     & b^{(0)}_{1,n} & 1      & 0                  &           &          &    \\
          \vdots            &               & \ddots & 1                  & \ddots    &          &  \\
          b^{(m-2)}_{m-2,n} &               &        & \ddots             & \ddots    & \ddots   &  \\
           0                & \ddots        &        &                    &  \ddots   & \ddots   & 0 \\
                            & \ddots        & \ddots &                    &           & \ddots   & 1  \\
                            &               &  0     & b^{(m-2)}_{k-1,n}  &  \ldots   & \ldots   & b^{(0)}_{k-1,n}
        \end{pmatrix}.
\end{equation}
The zeros of $P_{k,n}$ are equal to the eigenvalues of $M_{k,n}$, and their absolute values
are bounded by the maximum absolute row sum of $M_{k,n}$.
Therefore, by the definition \eqref{definitionR}, if $k\leq (s+1)n$ then the zeros of
$P_{k,n}$ lie in the interval $[-R,R]$.

We consider the family of functions
\begin{equation} \label{family}
    \mathcal H = \left\{ \frac{P_{k+1,n}}{P_{k,n}} \mid \, k,n\in \mathbb{N},\, k+1 \leq (s+1)n\right\}.
\end{equation}
Because of the assumption (a) in Theorem~\ref{Theorem_zero_distribution_general}
we can apply Lemma~\ref{lemma_polynomials_interlacing} to $P_{k,n}$ and $P_{k+1,n}$.
It follows from \eqref{inequality1}, \eqref{inequality2} and \eqref{family} that
the family $\mathcal H$ is a normal family (in the sense of Montel) on
$\mathbb C \setminus [-R,R]$.

Using induction on $l$, we will show the following.
\begin{description}
\item[Claim]
For each $l \geq 0$, the following holds.
If $\{ k_i \}_i$, $\{ n_i \}_i$ are sequences of non-negative integers with
 $k_i,n_i \to \infty$, $k_i/n_i \to s$ as $i\to \infty$, so that
\begin{equation} \label{deffx}
    f(x) :=  \lim_{i \to \infty} \frac{P_{k_i+1,n_i}(x)}{P_{k_i,n_i}(x)}
    \end{equation}
exists for $|x| > R$, then
\begin{equation} \label{inductionstatement}
    f(x) = z_1(x,s) + \mathcal O(x^{-l})
    \end{equation}
as $x \to \infty$.
\end{description}

We have $z_1(x,s) = x + \mathcal O(1)$, as $x \to \infty$, and so it is clear that
\eqref{inductionstatement} holds for $l = 0$.

Now assume that the claim holds for $l \geq 0$. Let $\{ k_i \}$ and $\{ n_i\}$
be as in the claim. Since $k_i/n_i \to s$ as $i \to \infty$, we may
assume that $k_i \leq (s+1) n_i$ for every $i$.  For $j = 0, \ldots, m-2$,
we then have that
\[ \frac{P_{k_i + 1- j,n_i}}{P_{k_i-j,n_i}} \]
belongs to the family $\mathcal H$. Since it is a normal family,
we may assume, by passing to a subsequence if necessary, that
\begin{equation} \label{definitionfj}
    f^{(j)}(x) = \lim_{i \to \infty} \frac{P_{k_i+1- j,n_i}(x)}{P_{k_i-j,n_i}(x)}
    \end{equation}
exists for $x \in \overline{\mathbb C} \setminus [-R,R]$ and $j = 0, \ldots, m-2$.

Now we divide the recurrence \eqref{recurrence_m} by $P_{k,n}$,
and replace $k$ and $n$ by $k_i$ and $n_i$, to obtain
for each $j$,
\[ x = \frac{P_{k_i + 1,n_i}(x)}{P_{k_i,n_i}(x)}
    + b_{k_i,n_i}^{(0)} +
        \sum_{j=1}^{m-2} b_{k_i,n_i}^{(j)} \frac{P_{k_i-j,n_i}(x)}{P_{k_i,n_i}(x)}.
      \]
We let $i \to \infty$, where we note that by \eqref{Convergence_coefficients_m}
\[ b_{k_i,n_i}^{(j)} \to b^{(j)}(s) \]
and by \eqref{definitionfj}
\[ \frac{P_{k_i-j,n_i}(x)}{P_{k_i,n_i}(x)} \to \left[f^{(1)}(x) \cdots f^{(j)}(x) \right]^{-1} \]
for $x \in \overline{\mathbb C} \setminus [-R,R]$ and $j = 1, \ldots, m-2$.
Thus
\[ x = f(x) + b^{(0)}(s) + \sum_{j=1}^{m-2} b^{(j)}(s) \left[f^{(1)}(x) \cdots f^{(j)}(x) \right]^{-1}. \]
for all $x\in \mathbb{C}\setminus [-R,R]$.

Now by the induction hypothesis we have for $j = 0, \ldots, m-2$,
\[ f^{(j)}(x) = z_1(x,s)  + \mathcal O(x^{-l}) = z_1(x,s) \left(1 + \mathcal O(x^{-l-1})\right)
    \qquad \text{as } x \to \infty. \]
and so
\begin{align*}
    f(x) & = x - b^{(0)}(s) - \sum_{j=1}^{m-2} b^{(j)}(s) \left[f^{(1)}(x) \cdots f^{(j)}(x) \right]^{-1} \\
    & = x - b^{(0)}(s) - \sum_{j=1}^{m-2} b^{(j)}(s) z_1(x,s)^{-j} \left(1 + \mathcal O(x^{-l-1})\right) \\
    & = x - b^{(0)}(s) - \sum_{j=1}^{m-2} b^{(j)}(s) z_1(x,s)^{-j} + \mathcal O(x^{-l-2}).
    \end{align*}
Since $z_1(x,s)$ is a solution of $A_s(z) = x$, we have
\[ x - b^{(0)}(s) - \sum_{j=1}^{m-2} b^{(j)}(s) z_1(x,s)^{-j} = z_1(x,s) \]
and so we obtain
\[ f(x)= z_1(x,s) + \mathcal{O}(x^{-l-2}) \]
as $x \to \infty$, which proves \eqref{inductionstatement} for $l+2$.

The claim now proved, we finally show how the lemma follows from the claim.
First note that $f$ is analytic in $\mathbb C \setminus [-R,R]$
and $x \mapsto z_1(x,s)$ is defined, analytic and non-zero in $\mathbb C \setminus \Gamma_1(s)$.
Thus, if \eqref{inductionstatement} holds for every $l$, then
clearly $f(x) = z_1(x,s)$ for $x$ in a neighborhood of $\infty$
in the $x$-plane, and by analyticity the equality extends to
$\mathbb C \setminus ([-R,R] \cup \Gamma_1(s))$.
Next, recall the assumption (b) in
Theorem~\ref{Theorem_zero_distribution_general}, which says that
$\Gamma_1(s) \subset \mathbb R$. Then it easily follows by its
definition that $z_1(x,s)$ cannot possibly have an analytic
continuation from $\mathbb C \setminus \Gamma_1(s)$ to a larger set.
Thus $\Gamma_1(s) \subset [-R,R]$ and it follows that
\begin{equation}
    \lim_{i \to \infty} \frac{P_{k_i+1,n_i}(x)}{P_{k_i,n_i}(x)}
    = z_1(x,s), \qquad x \in \mathbb C \setminus [-R,R],
\end{equation}
for all sequences $\{k_i\}$, $\{n_i\}$ as in the statement of \eqref{family}.

Then by a standard normal families argument (recall that the family $\mathcal H$
is a normal family) \eqref{asymptotic_dist_of_zeros} follows.

\end{proof}

\paragraph{Proof of Theorem~\ref{Theorem_zero_distribution_general}.} Let
$\xi > 0$, and take $0 \leq s \leq 1$. From Lemma
\ref{theorem_before_zeros_polynomials}, there exists $R>0$ such
that all zeros of $P_{k,n}$ belong to $[-R,R]$. For $k \in \mathbb
N$, we use $\lfloor sk\rfloor$ to denote the greatest integer less
than or equal to $sk$. We observe that
\begin{align} \nonumber
    \frac{1}{k} \frac{P'_{k,n}(x)}{P_{k,n}(x)} &
    =\frac1k \sum_{j=0}^{k-1}\left( \frac{P'_{j+1,n}(x)}{P_{j+1,n}(x)} - \frac{P'_{j,n}(x)}{P_{j,n}(x)}  \right) \\
    & = \int_0^1 \left( \frac{P'_{\lfloor sk\rfloor+1,n}(x)}{P_{\lfloor sk\rfloor+1 ,n}(x)}
    - \frac{P'_{\lfloor sk\rfloor,n}(x)}{P_{\lfloor sk\rfloor,n}(x)}  \right) ds.
    \label{riemman_sum}
\end{align}
By taking the logarithmic derivative of \eqref{asymptotic_dist_of_zeros} and
using \eqref{mu_1_identity_1} we obtain
\begin{equation}
\label{eq:log_derivative}
 \lim_{k/n\to \xi} \left( \frac{P'_{\lfloor sk\rfloor+1,n}(x)}{P_{\lfloor sk\rfloor+1 ,n}(x)}
    - \frac{P'_{\lfloor sk\rfloor,n}(x)}{P_{\lfloor sk\rfloor,n}(x)}  \right) =
    \frac{z'_1(x,s\xi)}{z_1(x,s\xi)}= \int \frac{d \mu^{s\xi}_1(y)}{x-y},
\end{equation}
uniformly on compact subsets of $\mathbb{C}\setminus[-R,R]$.
From \eqref{inequality2} and \eqref{inequality3} in Lemma~\ref{lemma_polynomials_interlacing}
we obtain
\begin{align}
\left| \frac{P'_{\lfloor sk\rfloor+1,n}(x)}{P_{\lfloor sk\rfloor+1 ,n}(x)} -
    \frac{P'_{\lfloor sk\rfloor,n}(x)}{P_{\lfloor sk\rfloor,n}(x)} \right|
     &=\left| \left(\frac{P_{\lfloor sk\rfloor,n}(x)}{P_{\lfloor sk\rfloor+1,n}(x)}\right)' \right|
    \left|\frac{P_{\lfloor sk\rfloor+1,n}(x)}{P_{\lfloor sk\rfloor,n}(x)}\right|\\
     &\leq \frac{2|z|}{\dist(z,[-R,R])^2} \nonumber.
\end{align}
for $|z|>R$. Therefore we may apply Lebesgue's dominated convergence theorem
and it follows by \eqref{riemman_sum} and \eqref{eq:log_derivative}
$$ \lim_{k/n \to \xi} \frac1k
\frac{P'_{k,n}(x)}{P_{k,n}(x)}=\int_0^1 \int \frac{d \mu^{s\xi}_1(y)}{x-y}ds =\frac{1}{\xi}\int_0^\xi \int \frac{d \mu^{s}_1(y)}{x-y} ds.$$
Therefore we have that
$$\lim_{k/n \to \xi}\frac{1}{k}\frac{P'_{k,n}(x)}{P_{k,n}(x)}
=\int \frac{1}{\xi}\int_0^\xi\frac{d \mu^{s}_1(y)}{x-y} ds,$$
for all $x\in\mathbb{C}\setminus [-R,R]$. This gives, by a standard argument,
see \cite{ST}, that the polynomials $P_{k,n}$ have
$$\nu^\xi_1=\frac{1}{\xi}\int_0^\xi \mu^{s}_1(\lambda) ds,$$
as limiting zero distribution as $k/n\to \xi$. \qed

\section{Proof of Proposition~\ref{Interlacing_Bk}}
\label{interlacing}

In \cite{CVA1} it is proved that the weights $(w_1,w_2)$ from
\eqref{Besselweights} are such that there  exist discrete measures
$\sigma_1$ and $\sigma_2$ on $(-\infty,0]$ such that
\begin{align*}
\frac{w_2(x)}{w_1(x)}=x\int_{-\infty}^0 \frac{d\sigma_1(t)}{x-t},\quad
\frac{w_1(x)}{w_2(x)}=\int_{-\infty}^0 \frac{d\sigma_2(t)}{x-t}, \quad x>0.
\end{align*}
From this it is shown in \cite[Theorem 4]{CVA1} that the weights
$(w_1,w_2)$ form an AT-system on $[0,\infty)$, which means that
for any $n_1, n_2 \in \mathbb N$, the set of real valued functions
$$\mathcal{F}_{n_1,n_2} = \{ w_1(x), xw_1(x), \ldots, x^{n_1} w_1(x),w_2(x),xw_2(x),\ldots,x^{n_2} w_2(x)\},$$
is a Chebyshev system on $(0,\infty)$ i.e., every linear
combination $\sum_{k=1}^{n_1+n_2+2} a_k\varphi_k$, with
$\varphi_k\in \mathcal{F}_{n_1,n_2}$, $\varphi_k\neq \varphi_{k'}$
if $k\neq k'$, and $(a_1, \ldots ,a_{n+m+2})\neq (0,\dots,0)$, has
at most $n_1+n_2+1$ zeros in $(0,\infty)$.

A consequence of this result is that for every $n_1, n_2 \in
\mathbb N$ there is a unique monic polynomial $B_{n_1,n_2}$ of
degree $n_1+n_2$ satisfying the orthogonality conditions
\begin{equation} \label{eq:orthogonalityBn1n2}
    \int_0^{\infty} B_{n_1,n_2}(x) x^k w_j(x) dx = 0,
\qquad
    \text{for } k=0, 1, \ldots, n_j-1, \qquad j =1,2.
\end{equation}
The polynomial $B_{n_1,n_2}$ has exactly $n_1+n_2$ simple zeros in
$(0,\infty)$, see e.g.\ \cite{VAC}. Comparing
\eqref{eq:orthogonalityBn1n2}  with \eqref{Besselweights} we see
that $B_n = B_{n_1,n_2}$ with $n_1 = n_2 = n/2$ if $n$ is even,
and $n_1 = (n+1)/2$, $n_2 = (n-1)/2$ if $n$ is odd.

The interlacing of zeros now follows from Lemma 2.3 and Remark 2.1
of \cite{AKLR}. Indeed there it is shown that for any $n_1, n_2$
the zeros of $B_{n_1,n_2}$ are interlaced with those of
$B_{n_1+1,n_2}$ and with those of $B_{n_1,n_2+1}$.

For convenience of the reader, we give the proof of interlacing
following \cite{AKLR} (see also \cite[Lemma~5 and Corollary~1]{FIL}
for a different approach). We show that the zeros of
$B_{n_1,n_2}$ interlace with those of $B_{n_1+1,n_2}$, the proof
for $B_{n_1,n_2+1}$ being similar.

Let us consider a polynomial $P = aB_{n_1,n_2}+ bB_{n_1+1,n_2}$
with $(a,b)\neq(0,0)$. Let us assume that $P$ has a double real
zero at $\zeta \in \mathbb R$. Then
\begin{equation}
\label{polynomial_P_R} P(x)=aB_{n_1,n_2}(x) +bB_{n_1+1,n_2}(x)=
(x-\zeta)^2 R(x), \qquad \deg R\leq n_1+n_2 -1.
\end{equation}
From the orthogonality conditions \eqref{eq:orthogonalityBn1n2} it
follows that $R$ satisfies
\begin{equation}
\label{eq:Rorthogonality}
\begin{aligned}
    \int_0^\infty R(x) x^k (x-\zeta)^2 w_j(x)dx
     & = \int_0^{\infty} \left(aB_{n_1,n_2}(x) +bB_{n_1+1,n_2}(x)
     \right) w_j(x) dx
     \\
    & = 0 \qquad \text{ for } k= 0, \ldots, n_j-1, \quad j=1,2.
\end{aligned}
\end{equation}
Thus $R$ has multiple orthogonality conditions with respect to the
weights $(\widetilde{w}_1, \widetilde{w}_2)$ where
 \[ \widetilde w_1(x) = (x-\zeta)^2 w_1(x), \qquad \text{and} \qquad
    \widetilde w_2(x) = (x-\zeta)^2 w_2(x). \]
By the same reasoning as in \cite{CVA1} it follows that
$(\widetilde{w}_1, \widetilde{w}_2)$ is also an AT system on
$[0,\infty)$.  It follows that any monic polynomial $R$ that
satisfies \eqref{eq:Rorthogonality} should have degree $\geq
n_1+n_2$, and this is a contradiction.

Therefore the polynomial $P$ has only simple zeros in $\mathbb R$.
This means that for every fixed $x\in \mathbb{R}$, the linear
system
$$\begin{pmatrix} B_{n_1,n_2}(x) & B_{n_1+1,n_2}(x) \\ B'_{n_1,n_2}(x) & B'_{n_1+1,n_2}(x)
   \\ \end{pmatrix} \begin{pmatrix} a \\ b \\ \end{pmatrix}
 =\begin{pmatrix} 0 \\ 0 \\ \end{pmatrix} $$
has only the trivial solution $a=b=0$. Then the determinant is
non-zero
$$ B_{n_1,n_2}(x)B'_{n_1+1,n_2}(x) - B_{n_1+1,n_2}(x) B'_{n_1,n_2}(x) \neq 0,
    \qquad x \in \mathbb R, $$
which readily implies that $B_{n_1,n_2}$ has opposite signs at
consecutive zeros of $B_{n_1+1,n_2}$. It follows that between two
consecutive zeros of $B_{n_1+1,n_2}$ there is at least one zero of
$B_{n_1,n_2}$, and this proves  the interlacing of zeros. \qed

\begin{figure}
\centering
\begin{overpic}[width=11cm,height=6cm]{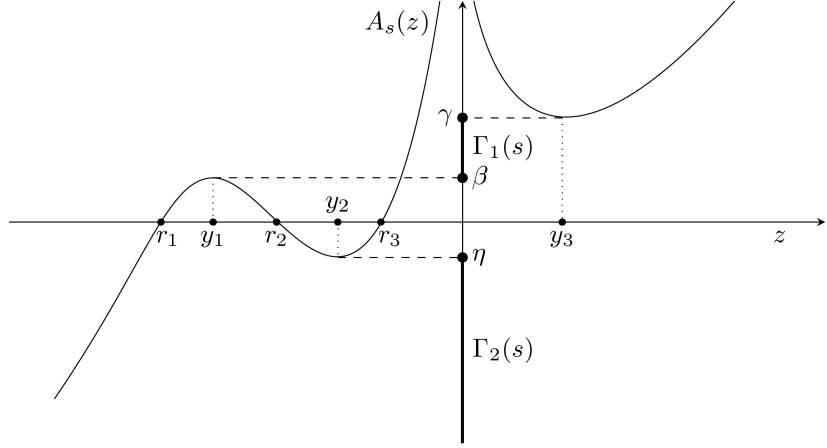}
\end{overpic}
\caption{The graph of $A_s(z)$ \eqref{family_of_symbols2},
$z\in \mathbb{R}$, for $p>0$.}\label{Plot_symbol_with_alpha}
\end{figure}

\section{Proof of Propositions~\ref{Proposition_Gamma_1_2} and \ref{Proposition_Gamma_1_2_alpha}}
\label{sec:proof_propositions}

Propositions~\ref{Proposition_Gamma_1_2} and
\ref{Proposition_Gamma_1_2_alpha} will be proved
simultaneously. First of all we observe that the symbol allows for
a factorization
\begin{equation}
\label{eq:factorization_As}
    A_s(z)=\frac{(z-r_1)(z-r_2)(z-r_3)}{z^2}, \qquad
    \text{with } r_1, r_2, r_3 < 0,
\end{equation}
see \eqref{family_of_symbols} for the case $p =0$, and
\eqref{family_of_symbols2}. We order the roots so that
\[ r_1 \leq r_2 \leq r_3 < 0. \]

If $p >0$, then it is elementary to check from
\eqref{family_of_symbols2} that all zeros are distinct.
Figure~\ref{Plot_symbol_with_alpha} shows the plot of $A_s(z)$ for the
case $p>0$. If $p=0$, then it follows from
\eqref{family_of_symbols} that
\begin{align*}
 r_1=r_2<r_3<0, & \qquad \text{ if } st<a(1-t),\\
 r_1=r_2=r_3<0, & \qquad \text{ if } st=a(1-t),\\
 r_1<r_2=r_3<0, & \qquad \text{ if } st>a(1-t).
\end{align*}
These three cases are illustrated in Figures~\ref{plotsA1},
\ref{plotsA2} and \ref{plotsA3}, respectively.

The derivative of $A_s(z)$,
$$A'_s(z) = 1 - c(s) z^{-2} - 2 d(s) z^{-3}, $$
has three roots in the complex plane. From Figures~\ref{Plot_symbol_with_alpha},
\ref{plotsA1}, \ref{plotsA2}, and
\ref{plotsA3} we see that all zeros are of $A_s'$ are real. We
denote the zeros of $A'(z)$ by $y_1$, $y_2$ and $y_3$ so that
$$ y_1 \leq y_2 < 0< y_3, $$
as indicated in the figures. The equality $y_1 = y_2$ only holds
if $p = 0$ and $s = s^*$ as in Figure~\ref{plotsA2}. To
indicate the dependence on $s$ we also write $y_1(s)$, $y_2(s)$
and $y_3(s)$.

Before proving Propositions  \ref{Proposition_Gamma_1_2} and
\ref{Proposition_Gamma_1_2_alpha} we first need two lemmas. The
proofs of these lemmas only use the fact that the zeros $r_j$ of
the symbol $A_s$ are strictly negative.

\begin{lem} \label{lem:x_in_R}
Assume that $z_1, z_2\in \mathbb{C}$ are such that $z_1\neq z_2$,
$|z_1|=|z_2|$ and $A_s(z_1)=A_s(z_2)=x$. Then $z_1=\bar z_2$ and
$x\in \mathbb{R}$.
\end{lem}
\begin{proof}
The complex numbers $z_1$ and $z_2$ lie in a circle of radius
$\rho = |z_1|=|z_2|$ centered at the origin in the complex plane.
Then from \eqref{eq:factorization_As} we have
$$ |A_s(z)|=\frac{\dist(z,r_1)\dist(z,r_2)\dist(z,r_3)}{\rho^2}, \quad \text{ if }|z|=\rho. $$
Since all $r_j<0$, it follows that
 \[ [-\pi, \pi] \to \mathbb R : \ \theta \mapsto |A_s(\rho e^{i\theta})| \]
is an even function which is  strictly decreasing as $\theta$
increases from $0$ to $\pi$.

Thus equality
 \[ |A_s(\rho e^{i \theta_1})| = | A_s(\rho e^{i \theta_2})|, \]
 with $\theta_{1,2} \in [-\pi, \pi]$ and $\theta_1 \neq \theta_2$
 can only occur if $\theta_2 = - \theta_1$. Then it follows
 from the assumptions of the lemma, that $z_1 = \bar{z}_2$,
 and
 \[ x=A_s(z_1)=\overline{A_s(\bar z_1)}=\overline{A_s(z_2)}=\overline{x} \]
 so that $x\in \mathbb{R}$.
\end{proof}
\begin{figure}
\centering
\begin{overpic}[width=6.7cm,height=6.7cm]{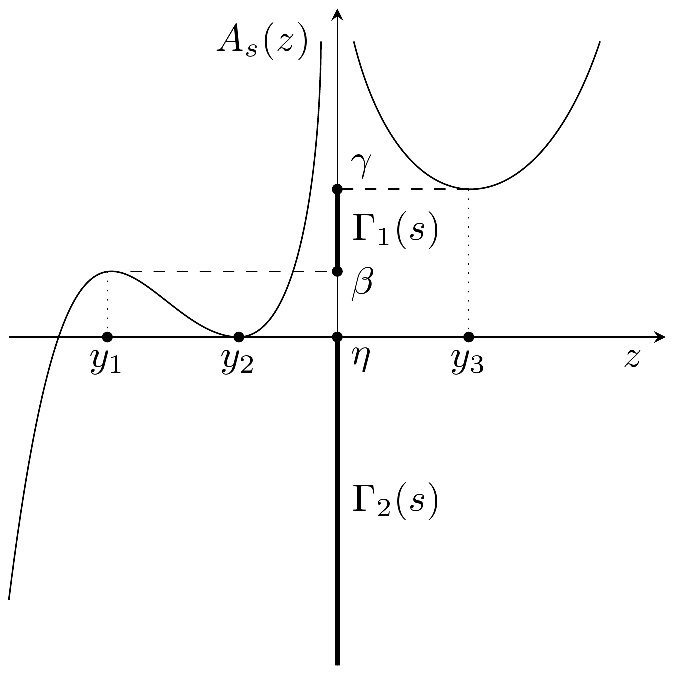}
\end{overpic}
\caption{Graph of $A_s(z)$ \eqref{family_of_symbols},
$z\in \mathbb{R}$, in the case $p=0$, $s<s^*$.}\label{plotsA1}
\end{figure}
\begin{lem}
\label{lem:Gammas_real}
For each $s>0$ we have $\Gamma_1(s)\cup\Gamma_2(s)\subset \mathbb{R}$. Moreover,
if $p>0$ or if $p=0$ and $s\neq s^*=a(1-t)/t$, then
$\Gamma_1(s)\cap\Gamma_2(s)=\emptyset$.  On the other hand, if $p=0$ and $s=s^*$,
then $\Gamma_1(s)\cap\Gamma_2(s)=\{0\}$.
\end{lem}
\begin{proof}
If $x\in \Gamma_1(s)\cup\Gamma_2(s)$ and $A_s(z)-x$ has a double root, then there exists
$z_1\in \mathbb{C}$ such that $A_s(z_1)=x$ and $A'_s(z_1)=0$. Since all zeros of $A'(z)$ are
real, it follows that $z_1\in \mathbb{R}$.

If $x\in \Gamma_1(s)\cup\Gamma_2(s)$ and $A_s(z)-x$ does not have a double root, then
there exist $z_1,z_2\in \mathbb{C}$ such that $z_1\neq z_2$, $|z_1|=|z_2|$ and
$x=A_s(z_1)=A_s(z_2)$. Lemma~\ref{lem:x_in_R} says that $x\in \mathbb{R}$.
This proves that $\Gamma_1(s)\cup\Gamma_2(s)\subset \mathbb{R}$.

If $x\in \Gamma_1(s)\cap\Gamma_2(s)$, and $s\neq s^*$ if $p=0$, then there exist three solutions
of $A_s(z)=x$. One negative solution $z_1<0$ and two complex conjugated solutions
$z_2$ and $\bar z_2$ (see Figures~\ref{Plot_symbol_with_alpha}, \ref{plotsA1}, \ref{plotsA3}).
Moreover $z_1\neq z_2$ and $z_1\neq \bar z_2$. Now $|z_1|=|z_2|$ ($x\in\Gamma_1(s)$) and
 $A_s(z_1)=A_s(z_2)$. Therefore we can apply Lemma
\ref{lem:x_in_R} and we obtain $z_1=\bar z_2$ which is a contradiction. Then
$\Gamma_1(s)\cap\Gamma_2(s)=\emptyset$.

The fact that $\Gamma_1(s^*)\cap\Gamma_2(s^*)=\{0\}$, if $p=0$, follows directly
from Figure~\ref{plotsA2}, by observing that $z=-a(1-t)^2$ is a triple root of $A_s(z)$.
\end{proof}

\paragraph{Proof of Propositions~\ref{Proposition_Gamma_1_2} and \ref{Proposition_Gamma_1_2_alpha}.}
If $p>0$ then there are three local extrema of $A_s(z)$, namely
$\beta(s), \gamma(s), \eta(s)$, such that
$$\eta(s)<0<\beta(s)<\gamma(s).$$
If $x\in(\eta(s),\beta(s))\cup(\gamma(s),\infty)$, then there exist three
different real solutions of $A_s(z)=x$. It is easily seen
that these solutions also differ in absolute value (Figure~\ref{Plot_symbol_with_alpha},
and Lemma~\ref{lem:x_in_R} if $x\in (\gamma(s),\infty))$.
On the other hand, there is one real and two complex conjugated solutions whenever
$x\in(-\infty,\eta(s)]\cup[\beta(s),\gamma(s)]$. Therefore
$$\Gamma_1(s)\cup\Gamma_2(s)\subset (-\infty,\eta(s)]\cup[\beta(s),\gamma(s)].$$
We have
that $A_s(z)=\gamma(s)$ has a double root $y_3=y_3(s)>0$ and one negative root whose absolute value
is less than $y_3(s)$. Thus $\gamma(s)\in \Gamma_1(s)$. Also $\beta(s)\in \Gamma_1(s)$.
Now the fact that $\Gamma_1(s)$ is connected (see \cite{Ullman},\cite[Theorem 11.19]{BG})
says that $\Gamma_1(s)=[\beta(s),\gamma(s)]$.
On the other hand $A_s(z)=\eta(s)$ has a double root at $y_2(s)<0$
and a negative root whose absolute value is larger than $|y_2(s)|$.
Therefore $\eta_2(s)\in \Gamma_2(s)$
and $\Gamma_2(s)=(-\infty,\eta(s)]$.

\begin{figure}
\centering
\begin{overpic}[width=6.7cm,height=6.7cm]{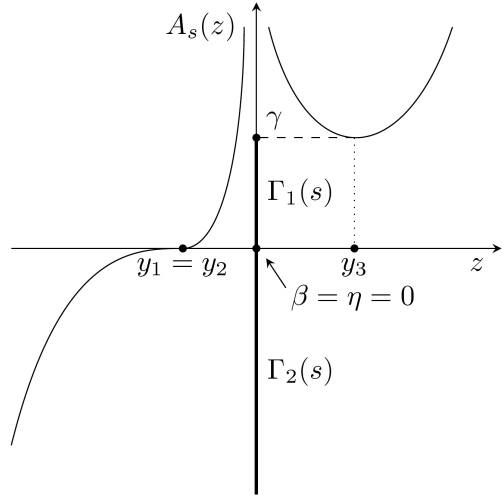}
\end{overpic}
\caption{Graph of $A_s(z)$ \eqref{family_of_symbols}, $z\in \mathbb{R}$, in the case $p=0$, $s=s^*$.}\label{plotsA2}
\end{figure}

The case $p=0$ follows analogously by studying Figures~\ref{plotsA1},
\ref{plotsA2}, and \ref{plotsA3} whenever
$st<a(1-t)$, $st=a(1-t)$ or $st>a(1-t)$, respectively. We deduce
from Figure~\ref{plotsA1} that $\eta(s)=0$ for $st<a(1-t)$. On the
other hand, Figure~\ref{plotsA3} shows that $\beta(s)=0$ if
$st>a(1-t)$.

We shall prove that $\gamma(s)$ is an increasing function. The monotonicity of $\beta(s)$
and $\eta(s)$ is proved with similar considerations. Observe that the function
$$B(z,s)=A_s(sz)=\frac{(sz+a(1-t)^2)(z+t(1-t))^2}{z^2}+\frac{p t(1-t)(z+t(1-t))}{z},$$
as a function of $z$, has a local minimum at $y_3(s)/s$ and
$$\gamma(s)=A_s(y_3(s))=B(y_3(s)/s,s).$$
If we take the partial derivative of $B(z,s)$ with respect to $s$ we obtain
$$\frac{\partial B(z,s)}{\partial s}=\frac{(z+t(1-t))^2}{z}.$$
which is positive for all $z>0$. Then the fact that $z\mapsto B(z,s)$ has a minimum
at $y_3(s)/s$ implies that
$$\gamma'(s)=\frac{\partial B(y_3(s)/s,s)}{\partial z}\frac{(y'_3(s)s-y_3(s))}{s^2}
+\frac{\partial B(y_3(s)/s,s)}{\partial s}=\frac{\partial B(y_3(s)/s,s)}{\partial s}>0,$$
and therefore $\gamma(s)$ is increasing for all $s>0$.

If $p>0$, it is a straightforward computation that $y_1(s)$, $y_2(s)$ and $y_3(s)$ have the
following behavior as $s\rightarrow \infty$
\begin{align*}
y_1(s)&=st(1-t)+\mathcal{O}(1),\\
y_2(s)&=-2a(1-t)^2+\mathcal{O}(s^{-1}),\\
y_3(s)&=-st(1-t)-\frac12 t(1-t)p+\mathcal{O}(s^{-1}).\\
\end{align*}
Then we obtain
\begin{align*}
\gamma(s)=A_s(y_3(s))&=4st(1-t)+\mathcal{O}(1),\\
\beta(s)=A_s(y_2(s))&=\frac14 tp^2(1-t)s^{-1}+\mathcal{O}(s^{-2}),\\
\eta(s)=A_s(y_1(s))&=-\frac{s^2t^2}{4a}+\mathcal{O}(s),
\end{align*}
as $s\to \infty$. This proves the limits
$$\lim_{s\rightarrow \infty} \beta(s)\rightarrow 0, \quad \lim_{s\rightarrow \infty}
\gamma(s)=\infty,\quad \lim_{s\rightarrow \infty}  \eta(s)=-\infty,$$
and completes the proof of the propositions. \qed

\begin{figure}
\centering
\begin{overpic}[width=6.7cm,height=6.7cm]{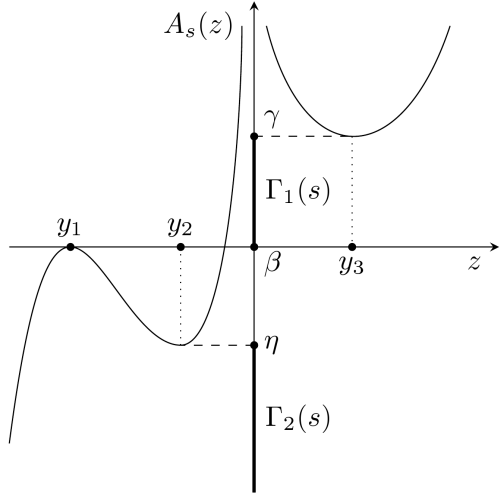}
\end{overpic}
\caption{ Plots of $A_s(z)$ \eqref{family_of_symbols}, $z\in \mathbb{R}$, in the case $p=0$, $s>a(1-t)/t$.}\label{plotsA3}
\end{figure}

\section{Proof of Theorem~\ref{proposition_V(x)_nu_1}}
\label{sec:proof_theorem_Var_conditions}

From the definitions of $\nu_1^{\xi}$ and $\nu_2^{\xi}$ in \eqref{nu1xi-as-integral}
and \eqref{nu2xi-as-integral}, it is immediate that $\supp(\nu_1^{\xi}) \subset [0,\infty)$,
$\int d\nu_1^{\xi} = 1$, $\supp(\nu_2^{\xi}) \subset (-\infty,0]$ and $\int d\nu_2^{\xi} = 1/2$.
In fact, from \eqref{nu1xi-as-integral} and the fact that the
sets $\Gamma_1(s) = \supp(\mu_1^s)$ are increasing as $s$ increases
(see Propositions~\ref{Proposition_Gamma_1_2} and \ref{Proposition_Gamma_1_2_alpha}) it
follows that
\begin{equation} \label{eq:supportnu1}
    \supp(\nu_1^\xi) = \bigcup_{s \leq \xi} \Gamma_1(s) = \Gamma_1(\xi).
    \end{equation}
From \eqref{nu2xi-as-integral} and the definition of $\sigma$ in \eqref{definition-of-sigma}, it is
then also clear that
\[ \nu^\xi_2 =\frac1\xi  \int_0^\xi \mu_2^sds \leq \frac1\xi\int_0^{\infty} \mu_2^sds
= \frac1\xi \sigma, \]
and
\begin{equation} \label{eq:sigmaminnu2}
    \sigma - \xi \nu_2^{\xi} = \int_{\xi}^{\infty} \mu_2^s ds
\end{equation}
so that
\begin{equation} \label{eq:supportnu2}
    \supp(\sigma - \xi \nu_2^{\xi}) = \bigcup_{s \geq \xi} \Gamma_2(s)
        = \Gamma_2(\xi)
        \end{equation}
where the last equality holds since the sets $\Gamma_2(s)$ are decreasing
as $s$ increases, see also Propositions~\ref{Proposition_Gamma_1_2}
and \ref{Proposition_Gamma_1_2_alpha}.

Thus in order to establish that $(\nu_1^{\xi}, \nu_2^{\xi})$ is the minimizer
of the energy functional \eqref{functionalwithV} under the conditions
stated in Theorem~\ref{proposition_V(x)_nu_1} it suffices to
prove that the variational conditions
\eqref{variational_nu_1} and \eqref{variational_nu_2} are satisfied.

The proof of \eqref{variational_nu_1} and \eqref{variational_nu_2} will be carried
out by integrating the variational conditions \eqref{variational_mu_1} and \eqref{variational_mu_2}
with respect to $s$ from $0$ to $\xi$. The proof of Theorem 2.3 of \cite{DK1}
contains a more general expression for the variational conditions \eqref{variational_mu_1} and
\eqref{variational_mu_2}, which is valid for any complex number $x$, namely
\begin{align}
\label{variational_mu_1_log}
2\int \log |x-y| d\mu^s_1(x) -  \int \log |x-y| d\mu^s_2(x) - \ell^s & = \log \left|
\frac{z_1(x,s)}{z_2(x,s)} \right|,\\
\label{variational_mu_2_log}
2\int \log |x-y| d\mu^s_2(x) -  \int \log |x-y| d\mu^s_1(x) &= \log \left|
\frac{z_2(x,s)}{z_3(x,s)} \right|.
\end{align}
These conditions reduce to
\eqref{variational_mu_1} and \eqref{variational_mu_2} whenever
$x\in \Gamma_1(s)$ and $x\in \Gamma_2(s)$, respectively.

\paragraph{Proof of \eqref{variational_nu_1}.}
If we multiply both sides of \eqref{variational_mu_1_log} by
$1/\xi$, integrate with respect to $s$ from $0$ to $\xi$, and
interchange the order of integration, then we obtain
\begin{equation} \label{eq:integrated1}
    2\int \log |x-y| d\nu^\xi_1(x) -  \int \log |x-y| d\nu^\xi_2(x) -
    \ell =\frac1\xi \int_0^\xi \log \left| \frac{z_1(x,s)}{z_2(x,s)}
    \right| ds,\quad x\in\mathbb{C},
    \end{equation}
for some constant $\ell\in \mathbb{R}$, where $\nu_1^\xi$ and
$\nu_2^\xi$ are the measures defined in \eqref{nu1xi-as-integral}
and \eqref{nu2xi-as-integral}, respectively.

Let $x > 0$. Since $|z_1(x,s)| \geq |z_2(x,s)|$ for every $s$ it follows
from \eqref{eq:integrated1} that
\begin{equation} \label{eq:integrated2}
    2\int \log |x-y| d\nu^\xi_1(x) -  \int \log |x-y| d\nu^\xi_2(x) -
    \ell \leq \frac{1}{\xi} \int_0^{\infty} \log \left| \frac{z_1(x,s)}{z_2(x,s)}
    \right| ds = \frac{1}{\xi} V(x).
    \end{equation}
If $x \in \supp(\nu_1^{\xi})$ then $x \in \Gamma_1(\xi)$ by \eqref{eq:supportnu1}
and therefore $x \in \Gamma_1(s)$ for every $s \geq \xi$,
since the sets are increasing. Thus $|z_1(x,s)| = |z_2(x,s)|$ for every $s \geq \xi$,
and equality holds in \eqref{eq:integrated2} for $x \in \supp(\nu_1^{\xi})$.
This completes the proof of \eqref{variational_nu_1}.

\paragraph{Proof of of \eqref{variational_nu_2}.}
If we multiply both sides of equation \eqref{variational_mu_2_log} by $1/\xi$, integrate
with respect to $s$ from $0$ to $\xi$, and we interchange the order of integration we obtain
\[  2\int \log |x-y| d\nu^\xi_2(x) -  \int \log |x-y| d\nu^\xi_1(x)
   = \frac1\xi \int_0^\xi \log \left| \frac{z_2(x,s)}{z_3(x,s)} \right| ds. \]
Since $|z_2(x,s)| \geq |z_3(x,s)|$ for every $s > 0$ it follows that
\begin{equation} \label{eq:equalitynu2}
    2\int \log |x-y| d\nu^\xi_2(x) -  \int \log |x-y| d\nu^\xi_1(x)
   \geq 0, \qquad x \in \mathbb C.
   \end{equation}
Equality holds in \eqref{eq:equalitynu2} if and only if
$|z_2(x,s)| = |z_3(x,s)|$ for every $s \in (0,\xi]$, that is, if and only if
\[ x \in \bigcap_{s \leq \xi} \Gamma_2(s) = \Gamma_2(\xi) = \supp(\sigma - \xi \nu_2^{\xi}), \]
where the first equality holds since the sets $\Gamma_2(s)$
are decreasing as $s$ increases, and the last equality holds because
of \eqref{eq:supportnu2}. This completes the proof of \eqref{variational_nu_2}.
\qed

\section{Proof of Theorem~\ref{thm:external_field}}
\label{sec:externalfield}

Before starting with the proof of  Theorem~\ref{thm:external_field} we establish the following lemma.
\begin{lem}
\label{asymptotic_solutions_alpha}
Let $A_s(z)$ be given by \eqref{family_of_symbols2}, and
let $z_1(x,s)$, $z_2(x,s)$ and $z_3(x,s)$ be the solutions of $A_s(z)=x$, ordered
as in \eqref{absolute_value_solutions}. Then
\begin{equation} \label{eq:zforsto0}
    \begin{aligned}
    \lim_{s\to 0+} z_1(x,s) & =x-p t (1-t) - a(1-t)^2,\\
    \lim_{s\to 0+} s^{-1}z_2(x,s) & =-\frac{2at(1-t)^2}{2a(1-t)+p t - \sqrt{p^2t^2+4ax}},\\
    \lim_{s\to 0+} s^{-1}z_3(x,s) & =-\frac{2at(1-t)^2}{2a(1-t)+p t - \sqrt{p^2t^2+4ax}}.
\end{aligned}
\end{equation}
\end{lem}
\begin{proof}
The lemma follows by a straightforward computation.
\end{proof}

For the proof of Theorem~\ref{thm:external_field} we need
to establish the two identities \eqref{external_field}
and \eqref{eq:density_constraint}.

\paragraph{Proof of \eqref{external_field}.}
Let $x > 0$. From Propositions~\ref{Proposition_Gamma_1_2} and \ref{Proposition_Gamma_1_2_alpha}
it follows that there exist a unique $s^*(x) \geq 0$ so that for all $s > 0$
\[ x \in \Gamma_1(s) \quad \Longleftrightarrow \quad s \geq s^*(x). \]
Then $\log |z_1(x,s)/z_2(x,s)|=0$ for all $s\geq s^*(x)$, and so by \eqref{definition-of-V}
\begin{equation} \label{eq:finiteintegralV}
     V(x)=\int_0^{s^*(x)}  \log \left| \frac{z_1(x,s)}{z_2(x,s)} \right| ds.
\end{equation}
There is a special value
\[ x_0 = (1-t) (a(1-t) + p t) = \lim_{s \to 0+} \beta(s) = \lim_{s \to 0+} \gamma(s) \]
that belongs to every $\Gamma_1(s)$ for all $s > 0$.
Then $s^*(x_0) = 0$ and
\begin{equation} \label{eq:Vinx0}
    V(x_0) = 0.
    \end{equation}

The derivative of \eqref{eq:finiteintegralV} is
\begin{equation} \label{integral_derivative_A}
    V'(x) = \int_0^{s^*(x)} \left(\frac{1}{z_1(x,s)}\frac{\partial z_1(x,s)}{\partial x}
    -\frac{1}{z_2(x,s)}\frac{\partial z_2(x,s)}{\partial x}\right) ds.
\end{equation}
In order to handle this integral we introduce new variables
\begin{equation}
\label{eq:variables_z_tilde}
    \widetilde{z}_1(x,s) = \frac{z_1(x,s)}{s}, \qquad
    \widetilde{z}_2(x,s) = \frac{z_2(x,s)}{s}, \qquad
    \widetilde{z}_3(x,s)=\frac{z_3(x,s)}{s}.
\end{equation}
Since $z_j(x,s)$ for $j=1,2,3$ is a  solution
 of $A_s(z)=x$, it follows that $\widetilde{z}_j(x,s)$ for $j=1,2,3$ is a solution of the
 equation $B(z,s) = x$ where
\begin{equation} \label{Symbol_B}
    B(z,s) = \frac{(sz+a(1-t)^2)(z+t(1-t))^2}{z^2}+\frac{p t(1-t)(z+t(1-t)}{z}.
\end{equation}

Taking partial derivatives with respect to $s$ and $x$ on both sides of $B(\widetilde{z}_j(x,s),s)=x$,
and applying the chain rule, we obtain
\begin{equation} \label{eq:Bpartials}
\begin{aligned}
    \left(\frac{\partial B}{\partial z}(\widetilde{z}_j(x,s),s) \right) \,
    \frac{\partial \widetilde{z}_j(x,s)}{\partial s} +
    \frac{\partial B}{\partial s}(\widetilde{z}_j(x,s),s) & = 0, \\
    \left(\frac{\partial B}{\partial z}(\widetilde{z}_j(x,s),s)\right)
    \frac{\partial \widetilde{z}_j(x,s)}{\partial x} & = 1
\end{aligned}
\end{equation}
for $j=1,2,3$. From \eqref{Symbol_B}, it is elementary to deduce that
\[ \frac{\partial B}{\partial s}(\widetilde{z}_j(x,s),s)
    =\frac{(\widetilde{z}_j(x,s)+t(1-t))^2}{\widetilde{z}_j(x,s)}. \]
Combining this with \eqref{eq:Bpartials}  we obtain
\begin{equation}
\label{eq:derivative_z_i}
    \frac{1}{\widetilde{z}_j(x,s)}\frac{\partial \widetilde{z}_j(x,s)}{\partial x}
    =-\frac{1}{(\widetilde{z}_j(x,s)+t(1-t))^2}\frac{\partial \widetilde{z}_j(x,s)}
    {\partial s}, \qquad j=1,2,3.
\end{equation}

Using \eqref{eq:derivative_z_i} in \eqref{integral_derivative_A} we get
\begin{equation*}
    V'(x)=-\int_0^{s^*(x)} \left(\frac{1}{(\widetilde{z}_1(x,s)+t(1-t))^2}
    \frac{\partial \widetilde{z}_1(x,s)}{\partial s}-\frac{1}{(\widetilde{z}_2(x,s)+t(1-t))^2}
    \frac{\partial \widetilde{z}_2(x,s)}{\partial s}\right) ds,
\end{equation*}
which can be written as
\begin{equation} \label{eq:Vprime2}
    V'(x) =
    \int_0^{s^*(x)} \frac{\partial}{\partial s}\big(F(\widetilde z_1(x,s))-F(\widetilde z_2(x,s))\big)ds
    \end{equation}
with
\begin{equation} \label{eq:def_F}
    F(x)=\frac{1}{x+t(1-t)}.
\end{equation}
From \eqref{eq:Vprime2} and the fundamental theorem of calculus, we have
\begin{equation} \label{eq:Vprime3}
    V'(x) = F(\widetilde z_1(x,s^*(x))-F(\widetilde z_2(x,s^*(x)) -\lim_{s\to 0+}
    (F(\widetilde z_1(x,s))-F(\widetilde z_2(x,s))).
    \end{equation}

By definition $s^*(x)$ is the smallest value of $s \geq 0$ for which $x\in \Gamma_1(s)$.
Then $x=\gamma(s^*(x))$ if $x_0<x$ and $x=\beta(s^*(x))$ if $0<x<x_0$.
We can observe from Figures~\ref{Plot_symbol_with_alpha}-\ref{plotsA3} that $z_1(\gamma(s),s)=z_2(\gamma(s),s)$ and
 $z_1(\beta(s),s)=z_2(\beta(s),s)$. Therefore
 \[ \widetilde z_1(x,s^*(x)) =\widetilde z_2(x,s^*(x)) \]
 and \eqref{eq:Vprime3} reduces because of \eqref{eq:def_F} to
\begin{equation} \label{eq:Vprime4}
    V'(x) = -\lim_{s\to 0+}
        \left(\frac{1}{\widetilde z_1(x,s)+t(1-t)}-\frac{1}{\widetilde z_2(x,s) +t(1-t)}
    \right).
    \end{equation}
From \eqref{eq:variables_z_tilde} and Lemma~\ref{asymptotic_solutions_alpha} we find that
\begin{align*}
    \lim_{s\rightarrow 0+} \widetilde z_1(x,s) = \infty
    \qquad \text{and} \qquad
    \lim_{s\rightarrow 0+} \widetilde z_2(x,s) =
    -\frac{2at(1-t)^2}{2a(1-t)+p t - \sqrt{p^2t^2+4ax}}.
\end{align*}
Then \eqref{eq:Vprime4} leads by straightforward computation to
\begin{equation} \label{eq:Vprime5}
V'(x)=\frac{1}{ t(1-t)}-\frac{2a}{t(\sqrt{p^2t^2+4ax}-p t)}.
\end{equation}

We  obtain $V(x)$ by integrating \eqref{eq:Vprime5} with respect to $x$. Thus
\[ V(x)=\frac{x}{t(1-t)}-\frac{\sqrt{p^2t^2+4ax}}{t}
-p \log(\sqrt{p^2t^2+4ax}-p t) + C. \]
The constant of integration $C$ should be such that $V(x_0) = 0$,
see \eqref{eq:Vinx0}.
This leads to
\[ C=\frac{a(1-t)}{t}+p\log(2a(1-t)) \]
and \eqref{external_field} is proved.

\paragraph{Proof of \eqref{eq:density_constraint}.}
 The measure $\sigma$,
introduced in \eqref{definition-of-sigma}, has the density
\begin{equation} \label{eq:sigmax1}
    \frac{d\sigma(x)}{dx} = \int_0^\infty \frac{d\mu_2^s(x)}{dx}ds, \qquad x < 0,
    \end{equation}
where $d\mu_2^s(x)/dx = 0$ for $x \not\in \Gamma_2(s)$ and by \eqref{eq:densitymu2}
\[  \frac{d\mu^s_2(x)}{dx}=
    \frac{1}{2\pi i} \left( \frac{z'_{2-}(x,s)}{z_{2-}(x,s)} -
    \frac{z'_{2+}(x,s)}{z_{2+}(x,s)} \right)
    =\frac{1}{2\pi i}
    \left( \frac{z'_{3+}(x,s)}{z_{3+}(x,s)} - \frac{z'_{2+}(x,s)}{z_{2+}(x,s)}
    \right),
\]
for $x\in \Gamma_2(s)$. The last equality comes from the
fact that $z_{2-}(x,s) = z_{3+}(x,s)$ for $x \in \Gamma_2(s)$.

From Proposition~\ref{Proposition_Gamma_1_2_alpha} it follows that
$d\mu_2^s(x)/dx=0$ for every $x \in(-p^2t^2/4a,0]$ and every $s > 0$ so
that
\[    \frac{d\sigma(x)}{dx} = 0 \qquad \text{for } x \in (-p^2 t^2/4a,0]. \]

Let $x < p^2 t^2/4a$. Again from Proposition~\ref{Proposition_Gamma_1_2_alpha}
(or from Proposition~\ref{Proposition_Gamma_1_2} in case $p = 0$)
it follows that there is a unique $s^*(x) > 0$ so that
\[ x \in \Gamma_2(s) \qquad \Longleftrightarrow \qquad s \leq s^*(x). \]
Then $d\mu_2^s(x)/dx=0$ for $s>s^*(x)$ so that  \eqref{eq:sigmax1} is in fact a finite integral
\begin{equation} \label{eq:sigmax2}
    \frac{d\sigma(x)}{dx} =
    \frac{1}{2\pi i} \int_0^{s^*(x)}
    \left( \frac{1}{z_{3+}(x,s)}\frac{\partial z_{3+}(x,s)}{\partial x}
    - \frac{1}{z_{2+}(x,s)}\frac{\partial z_{2+}(x,s)}{\partial x} \right) ds.
\end{equation}

The computation of \eqref{eq:sigmax2} is similar to
the computation of \eqref{integral_derivative_A} given above.
We use the functions $\widetilde{z}_j(x,s)$ as
in \eqref{eq:variables_z_tilde} and the function $F$
as defined in \eqref{eq:def_F}.
From \eqref{eq:derivative_z_i} and \eqref{eq:sigmax2} we then get
\begin{align} \nonumber
    2\pi i \frac{d\sigma(x)}{dx} & = - \int_0^{s^*(x)}
    \left(\frac{1}{(\widetilde{z}_{3+}+t(1-t))^2}\frac{\partial \widetilde{z}_{3+}}
    {\partial s}-\frac{1}{(\widetilde{z}_{2+}+t(1-t))^2}\frac{\partial \widetilde{z}_{2+}}{\partial s}\right)ds \\
    & = \nonumber
    F(\widetilde z_{3+}(x,s^*(x)))-F(\widetilde z_{2+}(x,s^*(x))) \\
    &\hspace{4cm}- \lim_{s\to 0+} \big(F(\widetilde z_{3+}(x,s))- F(\widetilde z_{2+}(x,s))\big) \nonumber\\
    & = \label{eq:sigmax3}
    -\lim_{s\to 0+} \big(F(\widetilde z_{3+}(x,s))-F(\widetilde z_{2+}(x,s))\big)
    \end{align}
since $\widetilde{z}_{2+}(x,s^*(x))=\widetilde{z}_{3+}(x,s^*(x))$.

Since $x < -p^2t^2/4a$, we obtain from \eqref{eq:variables_z_tilde}
and Lemma~\ref{asymptotic_solutions_alpha}
\begin{align*}
    \lim_{s\to 0+} \widetilde{z}_2(x,s)& = \lim_{s\rightarrow 0+} s^{-1}z_2(x,s)=
    -\frac{2at(1-t)^2}{2a(1-t)+p t - i\sqrt{4a|x|-p^2t^2}}, \\
    \lim_{s\to 0+} \widetilde{z}_3(x,s)& = \lim_{s\rightarrow 0+} s^{-1}z_3(x,s)=
    -\frac{2at(1-t)^2}{2a(1-t)+p t + i\sqrt{4a|x|-p^2t^2}}.
\end{align*}
Using this and \eqref{eq:def_F} in \eqref{eq:sigmax3} we easily get
\begin{align*}
    2\pi i \frac{d\sigma(x)}{dx} & = - \lim_{s\to 0+}
        \left(\frac{1}{\widetilde z_2(x,s)+t(1-t)}-\frac{1}{\widetilde z_3(x,s) +t(1-t)} \right)\\
    &=i \frac{\sqrt{4a|x|-p^2t^2 }}{t |x|}.
\end{align*}
which completes the proof of identity \eqref{eq:density_constraint}.
 \qed


\begin{thebibliography}{0}
\bibitem{AKLR}
 A.I. Aptekarev, V. Kalyagin, G. L\'opez Lagomasino, and I.A. Rocha:
 On the limit behavior of recurrence coefficients for multiple orthogonal polynomials,
 {\it J. Approx. Theory}  {\bf 139} (2006), 346--370.

\bibitem{BG}
 A. B\"ottcher and S.M. Grudsky:
 Spectral properties of banded Toeplitz matrices,
 SIAM, Philadelphia, PA, 2005.

\bibitem{CCVA}
 E. Coussement, J. Coussement, and W. Van Assche:
 Asymptotic zero distribution for a class of multiple orthogonal polynomials,
 {\it Trans. Amer. Math. Soc.}  {\bf 360}  (2008), 5571--5588.

\bibitem{CVA1}
 E. Coussement and W. Van Assche:
 Multiple orthogonal polynomials associated with the modified Bessel functions of the first kind,
 {\it Constr. Approx.} {\bf 19} (2003), 237--263.

\bibitem{CVA2}
 E. Coussement and W. Van Assche:
 Asymptotics of multiple orthogonal polynomials associated with the modified Bessel functions of the first kind,
 {\it J. Comput. Appl. Math.} {\bf 153} (2003), 141--149.

\bibitem{DK1}
 M. Duits and A.B.J. Kuijlaars:
 An equilibrium problem for the limiting eigenvalue distribution of banded Toeplitz matrices,
 {\it SIAM J. Matrix Anal. Appl.}  {\bf 30} (2008), 173--196.

\bibitem{DK2}
 M. Duits and A.B.J. Kuijlaars:
 Universality in the two matrix model: a Riemann-Hilbert steepest descent analysis,
 {\it Comm. Pure Appl. Math.} {\bf 62} (2009), 1076--1153.

\bibitem{FIL}
 U. Fidalgo Prieto, J. Ill\'an, and G. L\'opez Lagomasino:
 Hermite-Pad\'e approximation and simultaneous quadrature formulas,
 {\it J. Approx. Theory} {\bf 126} (2004), 171--197.

\bibitem{HP}
 F. Hiai and D. Petz:
 The semicircle law, free random variables and entropy,
 Amer. Math. Soc., Providence R.I., 2000.

\bibitem{Hirschman}
 I. I. Hirschman, Jr.:
 The spectra of certain Toeplitz matrices,
 {\it Illinois J. Math.} {\bf 11} (1967), 145--159.

\bibitem{Kuijlaars}
 A.B.J. Kuijlaars:
 Multiple orthogonal polynomial ensembles,
 preprint arxiv:0902.1058,
 to appear in Contemp. Math.

\bibitem{KMW}
 A.B.J. Kuijlaars, A. Mart\'inez Finkelstein and F. Wielonsky:
 Non-intersecting squared Bessel paths and multiple orthogonal polynomials for modified Bessel weights,
 {\it Commun. Math. Phys.}  {\bf 286} (2009), 217--275.

\bibitem{KS}
 A.B.J. Kuijlaars and S. Serra Capizzano,
 Asymptotic zero distribution of orthogonal polynomials with discontinuously varying recurrence coefficients,
 {\it J. Approx. Theory} {\bf 113} (2001), 142--155.

\bibitem{KVA}
 A.B.J. Kuijlaars and W. Van Assche:
 The asymptotic zero distribution of orthogonal polynomials with varying recurrence coefficients,
 {\it J. Approx. Theory} {\bf 99} (1999), 167--197.

\bibitem{ST}
 E.B. Saff and V. Totik:
 Logarithmic Potentials with External Fields,
 Springer-Verlag, Berlin, 1997.

\bibitem{SS}
 P. Schmidt and F. Spitzer:
 The Toeplitz matrices of an arbitrary Laurent polynomial,
 {\it Math. Scand.} {\bf 8} (1960), 15--38.

\bibitem{ShapT}
 B. Shapiro and M. Tater:
 On spectral polynomials of the Heun equation,
 preprint arXiv:0812.2321.

\bibitem{Ullman}
 J.L. Ullman:
 A problem of Schmidt and Spitzer,
 {\it Bull. Amer. Math. Soc.} {\bf 73} (1967), 883--885.

\bibitem{VAC}
 W. Van Assche and E. Coussement:
 Some classical multiple orthogonal polynomials,
 {\it J. Comput. Appl. Math.}  {\bf 127} (2001), 317--347.

\end{thebibliography}
\end{document}